\documentclass[11pt, oneside]{amsart}   	
\usepackage{amsmath,amsthm,amscd,amsfonts, amssymb}
\usepackage{geometry}                		
\usepackage{graphicx}				

\font\mbn=msbm10 scaled \magstep1
\font\mbs=msbm7 scaled \magstep1
\font\mbss=msbm5 scaled \magstep1
\newfam\mbff
\textfont\mbff=\mbn
\scriptfont\mbff=\mbs
\scriptscriptfont\mbff=\mbss

\newcommand{\RR}       { \mathbb{R}}

\newcommand{\cP}      {{\mathcal P}}
\newcommand{\N}       { \mathbb{N}}
\newcommand{\Z}        {\mathbb{Z}  }

\newtheorem{Th}{Theorem}[section]
\newtheorem{Lm}[Th]{Lemma}
\newtheorem{C}[Th]{Corollary}
\newtheorem{D}[Th]{Definition}
\newtheorem{Proposition}[Th]{Proposition}
\newtheorem{R}[Th]{Remark}

\begin{document}

\title[On Properties of Geometric Preduals of ${\mathbf C^{k,\omega}}$ Spaces]{On Properties of Geometric Preduals of ${\mathbf C^{k,\omega}}$ Spaces}
\author{Alexander Brudnyi} 
\address{Department of Mathematics and Statistics\newline
\hspace*{1em} University of Calgary\newline
\hspace*{1em} Calgary, Alberta\newline
\hspace*{1em} T2N 1N4}
\email{abrudnyi@ucalgary.ca}


\thanks{Research supported in part by NSERC}

\date{}							


\begin{abstract}
Let $C_b^{k,\omega}(\RR^n)$ be the Banach space of $C^k$ functions on $\RR^n$ bounded together with all derivatives of order $\le k$ and with derivatives of order $k$ having moduli of continuity majorated by $c\cdot\omega$, $c\in\RR_+$, for some $\omega\in C(\RR_+)$. Let $C_b^{k,\omega}(S):=C_b^{k,\omega}(\RR^n)|_S$  be the trace space to a closed subset $S\subset\RR^n$. The geometric predual $G_b^{k,\omega}(S)$ of $C_b^{k,\omega}(S)$ is the minimal closed subspace of the dual $\bigl(C_b^{k,\omega}(\RR^n)\bigr)^*$ containing evaluation functionals of points in $S$. We study geometric properties of spaces $G_b^{k,\omega}(S)$ and their relations to the classical Whitney problems on the characterization of trace spaces of $C^k$ functions on $\RR^n$. 
\end{abstract}

\maketitle

\section{Formulation of Main Results}
\subsection{Geometric Preduals of ${\mathbf C^{k,\omega}}$ Spaces} 
In what follows we use the standard notation of Differential Analysis. In particular,
$\alpha=(\alpha_1,\dots,\alpha_n)\in \mathbb Z^n_+$ denotes a multi-index
and $|\alpha|:=\sum^n_{i=1}\alpha_i$. Also, for $x=(x_1,\dots, x_n)\in\mathbb R^n$,
 \begin{equation}\label{eq1}
x^\alpha:=\prod^n_{i=1}x^{\alpha_i}_i \ \ \text{ and} \ \  D^\alpha:=\prod^n_{i=1}D^{\alpha_i}_i,\quad {\rm where}\quad D_i:=\frac{\partial}{\partial x_i}.
 \end{equation}


Let  $\omega$ be a nonnegative function on $(0,\infty)$ (referred to as {\em modulus of continuity})
satisfying the following conditions:
\begin{enumerate}
\item[(i)] $\omega(t)$ and $\displaystyle \frac {t}{\omega( t)}$ are nondecreasing functions on $(0,\infty)$;\medskip
\item[(ii)] $\displaystyle \lim_{t\rightarrow 0^+}\omega(t)=0$. 
\end{enumerate}

\begin{D}\label{def1}
$C^{k,\omega}_b(\mathbb R^n)$ is the Banach subspace of functions $f\in C^k(\mathbb R^n)$ with norm
 \begin{equation}\label{eq3}
\|f\|_{C^{k,\omega}_b(\mathbb R^n)}:=\max\left(\|f\|_{C^k_b(\mathbb R^n)}, |f|_{C^{k,\omega}_b(\mathbb R^n)}\right) ,
 \end{equation}
 where
 \begin{equation}\label{eq4}
 \|f\|_{C^k_b(\mathbb R^n)}:=\max_{|\alpha|\le k}\left\{\sup_{x\in\mathbb R^n}|D^\alpha f(x)|\right\}
 \end{equation}
and 
 \begin{equation}\label{eq5}
|f|_{C^{k,\omega}_b(\mathbb R^n)}:=\max_{|\alpha|=k}\left\{\sup_{x,y\in\mathbb R^n,\, x\ne y}\frac{|D^\alpha f(x)-D^\alpha
f(y)|}{\omega(\|x-y\|)}\right\}.
 \end{equation}
Here $\|\cdot\|$ is the Euclidean
norm of $\mathbb R^n$.
\end{D}
If $S\subset\mathbb R^n$ is a closed subset, then by $C_b^{k,\omega}(S)$ we denote the trace space of functions $g\in C_b^{k,\omega}(\mathbb R^n)|_{S}$ equipped with the quotient norm 
\begin{equation}\label{eq1.5}
\|g\|_{C_b^{k,\omega}(S)}:=\inf\{\|\tilde g\|_{C_b^{k,\omega}(\mathbb R^n)}\, :\, \tilde g\in C_b^{k,\omega}(\RR^n),\ \tilde g|_{S}=g\}.
\end{equation}

Let $\bigl(C_b^{k,\omega}(\mathbb R^n)\bigr)^*$ be the dual of $C_b^{k,\omega}(\mathbb R^n)$. Clearly, each evaluation functional $\delta_x^0$ at  $x\in\mathbb R^n$ (i.e., $\delta_x^0(f):=f(x)$, $f\in C_b^{k,\omega}(\RR^n)$) belongs to $\bigl(C_b^{k,\omega}(\mathbb R^n)\bigr)^*$ and has norm one. By $G^{k,\omega}_b(S)\subset \bigl(C_b^{k,\omega}(\mathbb R^n)\bigr)^*$ we denote the minimal closed subspace containing all $\delta_x^0$, $x\in S$. 

\begin{Th}\label{te1.2}
The restriction map to the set $\{\delta_s^0\, :\, s\in S\}\subset G_b^{k,\omega}(S)$ determines an isometric isomorphism  between the dual of $G^{k,\omega}_b(S)$ and $C^{k,\omega}_b(S)$.
\end{Th}
In what follows, $G_b^{k,\omega}(S)$ will be referred to as the {\em geometric predual} of $C^{k,\omega}_b(S)$.  In the present paper we study some properties of these spaces. The subject is closely related to the classical Whitney problems, see \cite{W1, W2},  asking about the characterization of trace spaces of $C^k$ functions on $\RR^n$ (see survey \cite{F3} and book \cite{BB2} and references therein for recent developments in the area). Some of the main results of the theory can be reformulated in terms of certain geometric characteristics of spaces $G_b^{k,\omega}(S)$, see sections~1.2 and 1.3.

\subsection{Finiteness Principle} For $C_b^{k,\omega}(S)$ this principle was conjectured by Yu.~Brudnyi and P.~Shvartsman in the 1980s in the 
following form (cf. \cite[p.\,210]{F3}).\smallskip

\noindent \textsl{\large{Finiteness Principle}.}  To decide whether a given $f:S\rightarrow\RR$, $S\subset\RR^n$, extends to a function $F\in C_b^{k,\omega}(\RR^n)$, it is enough to look at all restrictions $f|_{S'}$, where $S'\subset S$ is  an arbitrary $d$-element subset. Here $d$ is an integer constant depending only on $k$ and $n$.\smallskip

More precisely, if $f|_{S'}$ extends to a function $F^{S'}\in C_b^{k,\omega}(\RR^n)$ of norm at most $1$ (for each $S'\subset S$ with at most $d$ elements), then $f$ extends to a function $F\in C_b^{k,\omega}(\RR^n)$, whose norm is bounded by a constant depending only on $k$ and $n$.\smallskip

For $k=0$ (the Lipschitz case) the McShane extension theorem \cite{McS} implies the Finiteness Principle with the optimal constant $d=2$. 
Also, for $n=1$ the Finiteness Principle is valid with the optimal constant $d=k+2$. The result is essentially due to Merrien  \cite{M}.

In the multidimensional case the Finiteness Principle was proved by Yu.~Brudnyi and Shvartsman for $k=1$ with the optimal constant $d=3\cdot 2^{n-1}$, see \cite{BS1}. In the early 2000s the Finiteness Principle was proved by C.~Fefferman for all $k$ and $n$ for regular moduli of continuity $\omega$ (i.e., $\omega(1)=1$), see \cite{F1}. The upper bound for the constant $d$ in  the Fefferman proof was reduced later to $d=2^{ k+n \choose k}$ by Bierstone and P. Milman \cite{BM} and independently and by a different method by Shvartsman \cite{S}. The obtained results (and the Finiteness Principle in general) admit the following reformulation in terms of geometric characteristics of closed unit balls $B_b^{k,\omega}(S)$ of $G_b^{k,\omega}(S)$. Specifically,
let $B_b^{k,\omega}(S;m)\subset B_b^{k,\omega}(S)$, $m\in\N$, be the balanced closed convex hull of the union of all finite-dimensional balls $B_b^{k,\omega}(S')\subset G_b^{k,\omega}(S')$, $S'\subset S$, ${\rm card}\, S'\le m$.

\begin{Th}\label{teo1.5}
There exist constants $d\in\N$ and  $c\in (1,\infty)$ such that
\[
B_b^{k,\omega}(S;d)\subset B_b^{k,\omega}(S)\subset c\cdot B_b^{k,\omega}(S;d).
\]
Here for $k=0$, $d=2$ (-\,optimal) and $c=1$, for $n=1$, $d=k+2$ (-\,optimal) and $c$ depends on $k$ only, for $k=1$, $d=3\cdot 2^{n-1}$ (-\,optimal) and $c$ depends on $k$ and $n$ only, and for $k\ge 2$, $d=2^{ k+n \choose k}$ and $c=\frac{\tilde c}{\omega(1)}$, where $\tilde c$ depends on $k$ and $n$ only.
\end{Th}

\subsection{Complementability of Spaces ${\mathbf G_b^{k,\omega}(S)}$} We begin with a result describing bounded linear operators on $G_b^{k,\omega}(\RR^n)$. To this end,
for a Banach space $X$ by $C_b^{k,\omega}(\RR^n;X)$ we denote the Banach space of $X$-valued $C^k$ functions on $\RR^n$ with norm defined similarly to that of Definition \ref{def1} with absolute values replaced by norms $\|\cdot\|_X$ in $X$. 
Let $\mathcal L\bigl(X_1;X_2\bigr)$ stand for the Banach space of bounded linear operators between Banach spaces $X_1$ and $X_2$ equipped with the operator norm.
\begin{Th}\label{te1.6}
The restriction map to the set $\{\delta_x^0\, :\, x\in \RR^n\}\subset G_b^{k,\omega}(\RR^n)$ determines an isometric isomorphism between  $\mathcal L\bigl(G_b^{k,\omega}(\RR^n);X\bigr)$ and $C_b^{k,\omega}(\RR^n;X)$.
\end{Th}
 Let $q_S: C_b^{k,\omega}(\RR^n)\rightarrow C_b^{k,\omega}(S)$ be the quotient map induced by the restriction of functions on $\RR^n$ to $S$. A right inverse $T\in {\mathcal L}(C_b^{k,\omega}(S); C_b^{k,\omega}(\RR^n))$ for $q_S$ (i.e., $q_S\circ T={\rm id}$) is called a {\em linear extension operator}. The set of such operators is denoted by $Ext(C_b^{k,\omega}(S); C_b^{k,\omega}(\RR^n))$.
\begin{D}\label{def1.5}
An operator $T\in Ext(C_b^{k,\omega}(S); C_b^{k,\omega}(\RR^n))$ has depth $d\in\N$ if for all $x\in\RR^n$ and $f\in C_b^{k,\omega}(S)$, 
\begin{equation}\label{equ1.6}
(Tf)(x)=\sum_{i=1}^d \lambda_i^x\cdot f(y_i^x),
\end{equation}
where $y_i^x\in S$ and $\lambda_i^x$ depend only on $x$. 
\end{D}
Linear extension operators of finite depth exist. For $k=0$ (the Lipschitz case) the  Whitney-Glaeser linear extension operators $C_{b}^{0,\omega}(S)\rightarrow C_{b}^{0,\omega}(\RR^n)$, see \cite{Gl}, have depth $d$ depending on $n$ only and norms bounded by a constant depending on $n$ only.
In the 1990s bounded linear extension operators $C_{b}^{1,\omega}(S)\rightarrow C_{b}^{1,\omega}(\RR^n)$ of depth $d$ depending on $n$ only with norms bounded by a constant depending on $n$ only were constructed by Yu.~Brudnyi and Shvartsman \cite{BS2}.  Recently bounded linear extensions operators of depth $d$ depending on $k$ and $n$ only 
were constructed by Luli \cite{Lu} for all spaces $C_b^{k,\omega}(S)$; their norms are bounded by $\frac{C}{\omega(1)}$, where $C\in (1,\infty)$ is a constant depending on $k$ and $n$ only.
(Earlier such extension operators were constructed  for finite sets $S$ by C.~Fefferman \cite[Th.\,8]{F2}.)

In the following result we identify $(G_b^{k,\omega}(S))^*$ with $C_b^{k,\omega}(S)$ by means of the isometric isomorphism of Theorem \ref{te1.2}.
\begin{Th}\label{teo1.6}
For each $T\in Ext(C_b^{k,\omega}(S); C_b^{k,\omega}(\RR^n))$ of finite depth there exists a bounded linear projection $P:G_b^{k,\omega}(\RR^n)\rightarrow G_b^{k,\omega}(S)$ whose adjoint $P^*=T$.
\end{Th}
\begin{R}\label{rem1.7}
{\rm
It is easily seen that if $T$ has depth $d$ and is defined by  \eqref{equ1.6}, then
\[
p(x):=P(\delta_x^0)=\sum_{i=1}^d \lambda_i^x\cdot \delta_{y_i^x}^0
\quad {\rm for\ all}\quad x\in\RR^n.
\]
Moreover, $p\in C_b^{k,\omega}(\RR^n; G_b^{k,\omega}(S))$ and has norm equal to $\|T\|$ by Theorem \ref{te1.6}.
}
\end{R}
\subsection{Approximation Property}

Recall that a Banach space $X$ is said to have the {\em approximation property}, if,
for every compact set $K\subset X$ and every $\varepsilon > 0$, there exists an operator $T : X\to X$ of
finite rank so that $\|Tx-x\|\le\varepsilon$ for every $x\in K$. 

Although it is strongly believed that the class of spaces with the approximation property
includes practically all spaces which appear naturally in analysis, it is not known yet even
for the space $H^\infty$ of bounded holomorphic functions on the open unit disk.  The first example of a space which fails to have the approximation property was
constructed by Enflo \cite{E}. Since Enflo's work several other examples of such spaces were
constructed, for the references see, e.g., \cite{L}. 

A Banach space has the $\lambda$-{\em approximation property}, $1\le\lambda<\infty$, if it has the approximation property with the approximating
finite rank operators of norm $\le\lambda$. A Banach space is said to have the {\em bounded approximation property}, if it has the $\lambda$-approximation property for some $\lambda$. If $\lambda=1$, then the space is said to have the {\em metric approximation property}.

Every Banach spaces with a basis has the bounded approximation property. Also, it is known that the approximation property does not imply the bounded approximation property, see \cite{FJ}. It was established by Pe\l czy\'nski \cite{P} that a separable Banach space has the bounded approximation property if and only if it is isomorphic to a complemented subspace of a separable Banach space with a basis.

Next, for Banach spaces $X,Y$ by  ${\mathcal F}(X,Y)\subset {\mathcal L}(X,Y)$ we denote the subspace of linear bounded operators of finite rank $X\to Y$.
Let us consider the trace mapping $V$ from the projective tensor product $Y^*\hat{\otimes}_\pi X\to {\mathcal F}(X,Y)^*$ defined by
\[
(Vu)(T)={\rm trace}(Tu),\quad\text{where}\quad u\in Y^*\hat{\otimes}_\pi X,\ T\in {\mathcal F}(X,Y),
\]
that is, if $u=\sum_{n=1}^\infty y_n^*\otimes x_n$, then $(Vu)(T)=\sum_{n=1}^\infty y_n^*(Tx_n)$. 

It is easy to see that $\|Vu\|\le \|u\|_\pi$. The $\lambda$-bounded approximation property of $X$ is equivalent to the fact that
$\|u\|_\pi\le\lambda\|Vu\|$ for all Banach spaces $Y$. This well-known result (see, e.g., \cite[page 193]{DF}) is essentially due to Grothendieck \cite{G}.

Our result concerning spaces $G_b^{k,\omega}(S)$ reads  as follows.
\begin{Th}\label{te1.3}
\begin{enumerate}
\item Spaces $G_b^{k,\omega}(\RR^n)$ have the $\lambda$-approximation property with  \penalty-10000 $\displaystyle \lambda=\lambda(k,n,\omega):=1+C\cdot\lim_{t\rightarrow\infty}\,\mbox{$\frac{1}{\omega(t)}$}$,
where $C$ depends on $k$ and $n$ only.
\item All the other spaces $G_b^{k,\omega}(S)$ have the $\lambda$-approximation property with \penalty-10000 $\lambda= C'\cdot\lambda(1,n,\omega)$, where $C'$ is a constant depending on $n$ only, if $k=0,1$, and with $\lambda=\frac{C''\cdot \lambda(k,n,\omega)}{\omega(1)}$, where $C''$ is a constant depending on $k$ and $n$ only, if $k\ge 2$.
\end{enumerate}
\end{Th}
If $\lim_{t\rightarrow\infty}\omega(t)=\infty$, then (1) implies that the corresponding space $G_b^{k,\omega}(\RR^n)$ has the metric approximation property. In case $\lim_{t\rightarrow\infty}\omega(t)<\infty$, one can define the new modulus of continuity $\widetilde\omega$ (cf. properties (i) and (ii) in its definition) by the formula
\[
\widetilde\omega(t)=\max\{\omega(t),t\},\quad t\in (0,\infty).
\]
It is easily seen  that spaces $C_b^{k,\omega}(\RR^n)$ and $C_b^{k,\widetilde\omega}(\RR^n)$ are isomorphic. Thus $G_b^{k,\omega}(\RR^n)$ is isomorphic to space $G_b^{k,\widetilde\omega}(\RR^n)$ having the metric approximation property. However, the distortion of this isomorphism depends  on $\omega$. So, in general, it is not clear whether $G_b^{k,\omega}(\RR^n)$ itself has the metric approximation property.

In fact, in some cases spaces $G_b^{k,\omega}(S)$ still have the metric approximation property. E.g., by the classical result of Grothendieck \cite[Ch.\,I]{G},
separable dual spaces with the approximation property have the metric approximation property.
The class of such spaces $G_b^{k,\omega}(S)$ is studied in the next section.
\begin{R}\label{k}
{\rm It is not known, even for the case $k=0$, whether all spaces $C_b^{k,\omega}(\RR^n)$ have the approximation property (for some results in this direction for $k=0$ see, e.g., \cite{K}).
}
\end{R}
At the end of this section we formulate a result describing the structure of operators in ${\mathcal L}(G_b^{k,\omega}(\RR^n);X)$, where $X$ is a separable Banach space with the $\lambda$-approximation property. In particular, it can be applied to $X=G_b^{k,\omega}(S)$ and $\lambda:=\lambda(S,k,n,\omega)$ the constant of the approximation property for $G_b^{k,\omega}(S)$ of Theorem \ref{te1.3}\,(2).
\begin{Th}\label{teor1.10}
There exists the family of norm one vectors $\{v_j\}_{j\in\N}\subset X$ and given  $H\in {\mathcal L}(G_b^{k,\omega}(\RR^n);X)$ the family of functions $\{h_j\}_{j\in\N}\subset C_b^{k,\omega}(\RR^n)$ of norms $\le 32\cdot\lambda^2\cdot\|H\|$ such that for all $x\in\RR^n$,  $\alpha\in\Z_+^n$, $|\alpha|\le k$,
\begin{equation}\label{equa1.7}
H(\delta_x^\alpha)=\sum_{j=1}^\infty D^\alpha h_j(x)\cdot v_j
\end{equation}
(convergence in $X$).
\end{Th}
\begin{R}
{\rm If $X=G_b^{k,\omega}(S)$ and $H\in\mathcal L(G_b^{k,\omega}(\RR^n);G_b^{k,\omega}(S))$ is a projection onto $G_b^{k,\omega}(S)$, then in addition to \eqref{equa1.7} we have
\begin{equation}\label{equa1.8}
\delta_x^0=\sum_{j=1}^\infty  h_j(x)\cdot v_j\quad {\rm for\ all}\quad x\in S.
\end{equation}
In this case, the adjoint $H^*$ of $H$ belongs to $Ext(C_b^{k,\omega}(S);C_b^{k,\omega}(\RR^n))$ and for all $x\in\RR^n$,  $\alpha\in\Z_+^n$, $|\alpha|\le k$, the extension $H^*f$ of $f\in C_b^{k,\omega}(S)$ satisfies
\begin{equation}\label{equa1.9}
D^\alpha (H^*f)(x):=\sum_{j=1}^\infty D^\alpha h_j(x)\cdot f(v_j).
\end{equation}
}
\end{R}
\subsection{Preduals of ${\mathbf G_b^{k,\omega}(S)}$ Spaces} 
Let  $C_{0}^{k,\omega}(\RR^n)$ be the subspace of functions $f\in C_b^{k,\omega}(\RR^n)$ such that
\begin{itemize}
\item[(i)] for all $\alpha\in\Z_+^n$, $0\le |\alpha|\le k$,
\[
\lim_{\|x\|\rightarrow\infty}D^\alpha f(x)=0;
\]
\item[(ii)] for all $\alpha\in\Z_+^n$, $|\alpha|=k$,
\[
\lim_{\|x-y\|\rightarrow 0}\frac{D^\alpha f(x)-D^\alpha f(y)}{\omega(\|x-y\|)}=0.
\]
\end{itemize}
It is easily seen that $C^{k,\omega}_0(\RR^n)$ equipped with the norm induced from $C^{k,\omega}_b(\RR^n)$ is a Banach space. By  $C^{k,\omega}_0(S)$ we denote the trace of $C^{k,\omega}_0(\RR^n)$ to a closed subset $S\subset\RR^n$ equipped with the trace norm. 

If $\lim_{t\rightarrow 0^+}\,\frac{t}{\omega(t)}> 0$ (see condition (i) for $\omega$ in section~1.1), then clearly, the corresponding space $C^{k,\omega}_0(\RR^n)$ is trivial. Thus we may naturally assume that $\omega$ satisfies the condition
\begin{equation}\label{omega2}
\lim_{t\rightarrow 0^+}\,\frac{t}{\omega(t)}=0.
\end{equation}

In the sequel, the weak$^*$ topology of $C_b^{k,\omega}(S)$ is defined by means of functionals in $G_b^{k,\omega}(S)\subset  
\bigl(C_b^{k,\omega}(\RR^n)\bigr)^*$. Convergence in the weak$^*$ topology is described in section~4.2.
\begin{Th}\label{te1.4}
Suppose $\omega$ satisfies condition \eqref{omega2}. 
\begin{enumerate}
\item
Space
$\bigl(C^{k,\omega}_0(\RR^n)\bigr)^*$ is isomorphic to $G_b^{k,\omega}(\RR^n)$, isometrically if $\displaystyle \lim_{t\rightarrow\infty}\omega(t)=\infty$.
\item If there exists a weak$^*$ continuous operator 
$T\in Ext(C_b^{k,\omega}(S);C_b^{k,\omega}(\RR^n))$ such that $T\bigl(C_0^{k,\omega}(S)\bigr)\subset C_0^{k,\omega}(\RR^n)$, then $\bigl(C^{k,\omega}_0(S)\bigr)^*$ is isomorphic to $G_b^{k,\omega}(S)$.
\end{enumerate}
\end{Th}
From the first part of the theorem we obtain (for $\omega$ satisfying \eqref{omega2}):
\begin{C}\label{cor1.10}
The space of $C^\infty$ functions with compact supports on $\RR^n$ is dense in $C^{k,\omega}_0(\RR^n)$. In particular, all spaces $C^{k,\omega}_0(S)$ are separable.
\end{C}
It is not clear whether the condition of the second part of the theorem is valid for all spaces $C_b^{k,\omega}(S)$ with $\omega$ subject to \eqref{omega2}. Here we
describe a class of sets $S$ satisfying this condition.  As before, by $\mathcal P_{k,n}$ we denote the space of real polynomials on $\RR^n$ of degree $k$, and by $Q_r(x)\subset \RR^n$ the closed cube centered at $x$ of sidelength $2r$.
\begin{D}\label{wm}
A point $x$ of a subset $S\subset\RR^n$ is said to be weak $k$-Markov if 
\[
\varliminf_{r\rightarrow 0}\left\{\sup_{p\in\mathcal P_{k,n}\setminus 0}\left(\frac{\sup_{Q_r(x)}|p|}{\sup_{Q_r(x)\cap S}|p|}  \right)   \right\}<\infty .
\]

A closed set $S\subset\RR^n$ is said to be weak $k$-Markov if it contains a dense subset
of weak $k$-Markov points.
\end{D}
The class of weak $k$-Markov sets, denoted by ${\rm Mar}^*_k(\RR^n)$, was introduced and studied by Yu.~Brudnyi and the author, see
 \cite{BB1, B}. It contains, in particular, the closure of any open set, the Ahlfors $p$-regular compact subsets of $\RR^n$ with $p > n-1$, a wide class of fractals of arbitrary positive Hausdorff measure,  direct products $\prod_{j=1}^l S_j$, where $S_j\in {\rm Mar}^*_k(\RR^{n_j})$, $1\le j\le l$, $n=\sum_{j=1}^l n_j$, and closures of unions of any combination of such sets. Solutions of the Whitney problems (see sections 1.2 and 1.3 above) for sets in ${\rm Mar}^*_k(\RR^n)$ are relatively simple, see \cite{BB1}. 
 
We prove the following result.
 \begin{Th}\label{te1.11}
Let $S'\in {\rm Mar}^*_k(\RR^n)$ and $\omega$ satisfy \eqref{omega2}. Suppose $H:\RR^n\rightarrow\RR^n$ is a differentiable map such that 
\begin{itemize}
\item[(a)]
the entries of its Jacobian matrix belong to $C_b^{k-1,\omega_o}(\RR^n)$, where $\omega_o$ satisfies
\begin{equation}\label{equ1.8}
\lim_{t\rightarrow 0^+}\frac{\omega_o(t)}{\omega(t)}=0;
\end{equation}
\item[(b)]
the map $H|_{S'}:S'\rightarrow S=:H(S')$ is a proper retraction.\footnote{I.e.,  $S\subset S'$ and $H|_{S'}(x)=x$ for all $x\in S$, and for each compact $K\subset S$ its preimage $(H|_{S'})^{-1}(K)$ is compact.}
\end{itemize}

Then the condition of Theorem \ref{te1.4} holds for $C_b^{k,\omega}(S)$. Thus $G_b^{k,\omega}(S)$ is isomorphic to $\bigl(C^{k,\omega}_0(S)\bigr)^*$  and so $G_b^{k,\omega}(S)$ and $C^{k,\omega}_0(S)$ have the metric approximation property.
 \end{Th} 
 \begin{R}
 {\rm (1)
 In addition to weak $k$-Markov sets $S\subset\RR^n$, Theorem \ref{te1.11} is valid, e.g., for a compact subset $S$ of a $C^{k+1}$-manifold
 $M\subset\RR^n$ such that the base of the topology of $S$  consists of relatively open subsets of Hausdorff dimension $> {\rm dim}\,M - 1$. Indeed, in this case there exist tubular open neighbourhoods $U_M\subset V_M\subset\RR^n$ of $M$ such that ${\rm cl}(U_M)\subset V_M$ together with a $C^{k+1}$ retraction $r: U_M\rightarrow M$. Then, due to the hypothesis for $S$,  the base of topology of $S':=r^{-1}(S)\cap {\rm cl}(U_M) $ consists of relatively open subsets
of Hausdorff dimension $>n-1$ and so $S'\in {\rm Mar}_p^*(\RR^n)$ for all $p\in\N$, see, e.g., \cite[page\,536]{B}. Moreover, it is easily seen that $r|_{S'}$ is the restriction to $S'$ of a  map $H\in C_b^{k+1}(\RR^n; \RR^n)$. Decreasing $V_M$, if necessary, we may assume that $S'$ is compact, and so the triple $(H, S', S)$ satisfies the hypothesis of the theorem.
 
\noindent (2) Under conditions of Theorem \ref{te1.11},  $C_b^{k,\omega}(S)$ is isomorphic to the second dual of $C^{k,\omega}_0(S)$.}
 \end{R}
\section{Proof of Theorem \ref{te1.2}}
By $\delta_x^\alpha$, $x\in\RR^n$, $\alpha\in\mathbb Z^n_+$, we denote the evaluation functional $D^\alpha|_{\{x\}}$. By definition each $\delta_x^\alpha$, $|\alpha|\le k$, belongs to $\bigl(C^{k,\omega}_b(\RR^n)\bigr)^*$ and has norm $\le 1$. Similarly, functionals $\frac{\delta_x^\alpha-\delta_y^\alpha}{\omega(\|x-y\|)}$, $|\alpha|=k$, $x,y\in\RR^n$, $x\ne y$, belong to $\bigl(C^{k,\omega}_b(\RR^n)\bigr)^*$ and have norm $\le 1$.
\begin{Proposition}\label{p2.1}
The closed unit ball $B$ of $\bigl(C^{k,\omega}_b(\RR^n)\bigr)^*$ is the balanced weak$\,^*$ closed convex hull of the set $V$ of all functionals $\delta_x^\alpha$, $|\alpha|\le k$, and $\frac{\delta_x^\alpha-\delta_y^\alpha}{\omega(\|x-y\|)}$, $|\alpha|=k$, $x,y\in\RR^n$, $x\ne y$.
\end{Proposition}
\begin{proof}
Clearly, $V\subset B$ and therefore the required hull $\widehat V\subset B$ as well. Assume, on the contrary, that $\widehat V\ne B$. Then due to the Hahn-Banach theorem there exists an element $f\in C^{k,\omega}_b(\RR^n)$ of norm one such that $\sup_{v\in\widehat V}|v(f)|\le c<1$. Since $V\subset \widehat V$, this implies
\[
\|f\|_{C^{k,\omega}_b(\RR^n)}\le c<1,
\]
a contradiction proving the result.
\end{proof}
Let $X$ be the minimal closed subspace of $\bigl(C^{k,\omega}_b(\RR^n)\bigr)^*$ containing $V$.
\begin{Proposition}\label{p2.2}
$X^*$ is isometrically isomorphic to $C^{k,\omega}_b(\RR^n)$.
\end{Proposition}
\begin{proof}
For $h\in X^*$ we set $H(x):=h(\delta^0_x)$, $x\in\mathbb R^n$.
Let $e_1,\dots, e_n$ be the standard orthonormal basis in $\RR^n$. By the mean-value theorem for functions in $C^{k,\omega}_b(\RR^n)$ we obtain, for all $\alpha\in\mathbb Z^n_+$, $|\alpha|<k$,
$x\in\RR^n$,
\begin{equation}\label{eq2.6}
\lim_{t\rightarrow 0}\frac{\delta^{\alpha}_{x+t\cdot e_i}-\delta^\alpha_x}{t}=\delta^{\alpha+e_i}_x 
\end{equation}
(convergence in $\bigl(C^{k,\omega}_b(\RR^n)\bigr)^*)$. From here by induction we deduce easily that $H\in C^k(\RR^n)$ and for all $\alpha\in\mathbb Z^n_+$, $|\alpha|\le k$,
$x\in\RR^n$,
\[
h(\delta^\alpha_x)=D^\alpha H(x).
\]
This shows that $H\in C^{k,\omega}_b(\RR^n)$ and $\|H\|_{C^{k,\omega}_b(\RR^n)}\le \|h\|_{X^*}$. Considering $H$ as the bounded linear functional on $\bigl(C^{k,\omega}_b(\RR^n)\bigr)^*$ we obtain that
$H|_V=h|_V$. Thus, by the definition of $X$,
\[
H|_{X}=h.
\]
Since the unit ball of $X$ is $B\cap X$, 
\[
\|h\|_{X^*}\le \|H\|_{C^{k,\omega}_b(\RR^n)}\, \bigl(\le \|h\|_{X^*}\bigr).
\]
Hence, the correspondence $h\mapsto H$ determines an isometry $I :X^*\rightarrow C_b^{k,\omega}(\RR^n)$. Since the restriction of each  $H\in C_b^{k,\omega}(\RR^n)$, regarded as the bounded linear functional on $\bigl(C^{k,\omega}_b(\RR^n)\bigr)^*$, to $X$ determines some $h\in X^*$, map $I$ is surjective.

This completes the proof of the proposition.
\end{proof}
Note that equation \eqref{eq2.6} shows that the minimal closed subspace $G_b^{k,\omega}(\RR^n)\subset\bigl(C_b^{k,\omega}(\RR^n)\bigr)^*$ containing all $\delta_x^0$, $x\in\RR^n$, coincides with $X$. Thus, 
$\bigl(G_b^{k,\omega}(\RR^n)\bigr)^*$ is isometrically isomorphic to $C^{k,\omega}_b(\RR^n)$; this proves Theorem \ref{te1.2} for $S=\RR^n$.
\begin{C}\label{cor2.3}
The closed unit ball of $G^{k,\omega}_b(\RR^n)$ is the balanced closed convex hull of the set $V$ of all functionals $\delta_x^\alpha$, $|\alpha|\le k$, and $\frac{\delta_x^\alpha-\delta_y^\alpha}{\omega(\|x-y\|)}$, $|\alpha|=k$, $x,y\in\RR^n$, $x\ne y$.
\end{C}
\begin{proof}
The closed unit ball of $G^{k,\omega}_b(\RR^n)$ is
$B\cap G^{k,\omega}_b(\RR^n)$. Since the weak$^*$ topology of $\bigl(C^{k,\omega}_b(\RR^n)\bigr)^*$ induces the weak topology of $G^{k,\omega}_b(\RR^n)$ and the weak closure of the balanced convex hull of $V$ coincides with the norm closure of this set, the result follows from Proposition \ref{p2.1}.
\end{proof}
Now, let us consider the case of general $S\subset\mathbb R^n$. Let $h\in \bigl(G_b^{k,\omega}(S)\bigr)^*$. We set $H(x):=h(\delta_x^0)$, $x\in S$. Due to the Hahn-Banach theorem, there exists $\tilde h\in \bigl(G_b^{k,\omega}(\RR^n)\bigr)^*$ such that $\tilde h|_{G_b^{k,\omega}(S)}=h$ and $\|\tilde h\|_{(G_b^{k,\omega}(\RR^n))^*}=\|h\|_{(G_b^{k,\omega}(S))^*}$. Let us define $\widetilde H(x)=\tilde h(\delta_x^0)$, $x\in \RR^n$. According to Proposition \ref{p2.2}, 
$\widetilde H\in C^{k,\omega}_b(\RR^n)$ and $\|\widetilde H\|_{C_b^{k,\omega}(\RR^n)}=\|\tilde h\|_{(G_b^{k,\omega}(\RR^n))^*}$. Moreover, $\widetilde H|_S=H$. This implies that $H\in C_b^{k,\omega}(S)$ and has norm $\le \|h\|_{(G_b^{k,\omega}(S))^*}$. Hence, the correspondence $h\mapsto H$ determines a bounded nonincreasing norm linear injection $I_S:\bigl(G_b^{k,\omega}(S)\bigr)^*\rightarrow C_b^{k,\omega}(S)$.
Let us show that $I_S$ is a surjective isometry. Indeed, for $H\in C_b^{k,\omega}(S)$ there exists $\widetilde H\in C_b^{k,\omega}(\RR^n)$ such that $\widetilde H|_S=H$ and $\|\widetilde H\|_{C_b^{k,\omega}(\RR^n)}=\|H\|_{C_b^{k,\omega}(S)}$.  In turn, due to Proposition \ref{p2.2} there exists $\tilde h\in \bigl(G_b^{k,\omega}(\RR^n)\bigr)^*$ such that $\widetilde H(x)=\tilde h(\delta_x^0)$, $x\in\RR^n$, and $\|\widetilde H\|_{C_b^{k,\omega}(\RR^n)}=\|\tilde h\|_{(G_b^{k,\omega}(\RR^n))^*}$. We set $h:=\tilde h|_{G_b^{k,\omega}(S)}$. Then $h\in \bigl(G_b^{k,\omega}(S)\bigr)^*$ and $H(x)=h(\delta_x^0)$, $x\in S$, i.e.,
$I_S(h)=H$ and 
\[
\bigl(\|h\|_{(G_b^{k,\omega}(S))^*}\ge\bigr)\, \|I_S(h)\|_{C_b^{k,\omega}(S)}\ge \|h\|_{(G_b^{k,\omega}(S))^*}.
\]
 
The proof of Theorem \ref{te1.2} is complete.

\section{Proofs of Theorems \ref{teo1.5}, \ref{te1.6}, \ref{teo1.6}}
\subsection{Proof of Theorem \ref{teo1.5}}
\begin{proof}
According to the Finiteness Principle 
there exist constants $d\in\N$ and  $c\in (1,\infty)$ such that for all $f\in C_b^{k,\omega}(S)$,
\begin{equation}\label{e3.13}
\sup_{S'\subset S\,;\, {\rm card}\,S'\le d}\|f\|_{C_b^{k,\omega}(S')}\le\|f\|_{C_b^{k,\omega}(S)}\le c\cdot\left(\sup_{S'\subset S\,;\, {\rm card}\,S'\le d}\|f\|_{C_b^{k,\omega}(S')}\right).
\end{equation}
Here for $k=0$, $d=2$ (-\,optimal) and $c=1$, see \cite{McS}, for $n=1$, $d=k+2$ (-\,optimal) and $c$ depends on $k$ only, see \cite{M}, for $k=1$, $d=3\cdot 2^{n-1}$ (-\,optimal) and $c$ depends on $k$ and $n$ only, see \cite{BS1}, and for $k\ge 2$, $d=2^{ k+n \choose k}$ and $c=\frac{\tilde c}{\omega(1)}$, where $\tilde c$ depends on $k$ and $n$ only, see \cite{F1} and \cite{BM}, \cite{S}.

Considering $f$ as the bounded linear functional on $G_b^{k,\omega}(S)$, we get from \eqref{e3.13} the required implications
\[
B_b^{k,\omega}(S;d)\subset B_b^{k,\omega}(S)\subset c\cdot B_b^{k,\omega}(S;d).
\]
Indeed, suppose, on the contrary, that there exists $v\in  B_b^{k,\omega}(S)\setminus  c\cdot B_b^{k,\omega}(S;d)$. Let $f\in C_b^{k,\omega}(S)$ be such that
\[
\sup_{c\cdot B_b^{k,\omega}(S;d)}|f|<|f(v)|.
\] 
By the definition of $B_b^{k,\omega}(S;d)$ the left-hand side of the previous inequality coincides with $c\cdot\bigl(\sup_{S'\subset S\,;\, {\rm card}\,S'\le d}\,\|f\|_{C_b^{k,\omega}(S')}\bigr)$. Hence,
\[
c\cdot\left(\sup_{S'\subset S\,;\, {\rm card}\,S'\le d}\|f\|_{C_b^{k,\omega}(S')}\right)<|f(v)|\le \|f\|_{C_b^{k,\omega}(S)},
\]
a contradiction with \eqref{e3.13}.
\end{proof}
\subsection{Proof of Theorem \ref{te1.6}}
\begin{proof}
We set
\begin{equation}\label{e4.14}
r_{X}(F)(s):=F(\delta_s^0),\quad F\in \mathcal L(G_b^{k,\omega}(\RR^n); X),\quad s\in \RR^n.
\end{equation}
Applying the arguments similar to those of Proposition \ref{p2.2} we obtain
\[
r_{X}(F)\in C_b^{k,\omega}(\RR^n;X)\quad {\rm and}\quad  \|r_X(F)\|_{C_b^{k,\omega}(\RR^n;X)}\le \|F\|_{\mathcal L(G_b^{k,\omega}(\RR^n);X)}.
\]
On the other hand, for each $\varphi\in X^*$, $\|\varphi\|_{X^*}=1$, function $r_{\RR}(\varphi\circ F)\in C_b^{k,\omega}(\RR^n)$. So, since $r_{\RR}(\varphi\circ F)=\varphi (r_{X}(F))$, 
\[
\|\varphi\circ F\|_{(G_b^{k,\omega}(\RR^n))^*}=\|r_{\RR}(\varphi\circ F)\|_{C_b^{k,\omega}(\RR^n)}=\|\varphi (r_{\RR^n;X}(F))\|_{C_b^{k,\omega}(\RR^n)}\le \|r_{X}(F)\|_{C_b^{k,\omega}(\RR^n;X)}.
\]
Taking supremum over all such $\varphi$ we get
\[
\| F\|_{\mathcal L(G_b^{k,\omega}(\RR^n);X)}\le \|r_{X}(F)\|_{C_b^{k,\omega}(\RR^n;X)}\, \bigl(\le \| F\|_{\mathcal L(G_b^{k,\omega}(\RR^n);X)}\bigr).
\]
This shows that $r_{X}:\mathcal L(G_b^{k,\omega}(\RR^n);X)\rightarrow C_b^{k,\omega}(\RR^n;X)$ is an isometry. Let us prove that it is surjective.

Since every finite subset of $\RR^n$ is interpolating for $C_b^{k,\omega}(\RR^n)$,
 the set of vectors $\delta_s^0\in G_b^{k,\omega}(\RR^n)$, $s\in\RR^n$, is linearly independent. Hence, each $f\in C_b^{k,\omega}(\RR^n;X)$ determines a linear map $\hat f:{\rm span}\{\delta_s^0\, :\, s\in \RR^n\}\rightarrow X$,
 \[
\hat f\left(\sum_{j}c_j\delta_{s_j}^0\right):=\sum_j c_j f(s_j),\quad \sum_{j}c_j\delta_{s_j}^0\in {\rm span}\{\delta_s^0\, :\, s\in \RR^n\}\, (\subset G_b^{k,\omega}(\RR^n)).
\]
Next, for each $\varphi\in X^*$, $\|\varphi\|_{X^*}=1$,  function $\varphi\circ f\in C_b^{k,\omega}(\RR^n)$ and 
\[
\|\varphi\circ f\|_{C_b^{k,\omega}(\RR^n)}\le\|f\|_{C_b^{k,\omega}(\RR^n;X)}.
\]
Since $r_{\RR}:\bigl(G_b^{k,\omega}(\RR^n)\bigr)^*\rightarrow C_b^{k,\omega}(\RR^n)$ is an isometric isomorphism, there exists $\ell_{\varphi\circ f}\in \bigl(G_b^{k,\omega}(\RR^n)\bigr)^*$ such that $r_{\RR}(\ell_{\varphi\circ f})=\varphi\circ f$. Clearly, $\ell_{\varphi\circ f}$ coincides with $\varphi\circ\hat f$ on ${\rm span}\{\delta_s^0\, :\, s\in \RR^n\}$ and  for all $v\in G_b^{k,\omega}(\RR^n)$,

\[
\begin{array}{r}
\displaystyle
|\ell_{\varphi\circ f}(v)|\le \| \ell_{\varphi\circ f}\|_{(G_b^{k,\omega}(\RR^n))^*}\cdot \|v\|_{G_b^{k,\omega}(\RR^n)}
 =\|\varphi\circ f\|_{C_b^{k,\omega}(\RR^n)}\cdot \|v\|_{G_b^{k,\omega}(\RR^n)}\medskip\\

 \displaystyle \le \|f\|_{C_b^{k,\omega}(\RR^n;X)}\cdot \|v\|_{G_b^{k,\omega}(\RR^n)}.
\end{array}
\]
These imply that $\hat f:{\rm span}\{\delta_s^0\, :\, s\in \RR^n\}\rightarrow X$ is a linear continuous operator  of norm $\le \|f\|_{C_b^{k,\omega}(\RR^n;X)}$. Hence, it extends
to a bounded linear operator $F:{\rm cl}({\rm span}\{\delta_s^0\, :\, s\in \RR^n\})=:G_b^{k,\omega}(\RR^n)\rightarrow X$ such that $r_{X}(F)=f$. Thus,
 $r_{X}(F):\mathcal L(G_b^{k,\omega}(\RR^n);X)\rightarrow C_b^{k,\omega}(\RR^n;X)$ is an isometric isomorphism.
 
The proof of the theorem is complete.
\end{proof}
\subsection{Proof of Theorem \ref{teo1.6}}

\begin{proof}
Without loss of generality we may assume that $T$ has depth $d$ and is defined by \eqref{equ1.6}.
Let $T:\bigl(C_b^{k,\omega}(\RR^n)\bigr)^*\rightarrow \bigl(C_b^{k,\omega}(S)\bigr)^*$ be the adjoint of $T$ and $q_S^*:\bigl(C_b^{k,\omega}(S)\bigr)^*\rightarrow \bigl(C_b^{k,\omega}(\RR^n)\bigr)^*$ the adjoint of the quotient map $q_S: C_b^{k,\omega}(\RR^n)\rightarrow C_b^{k,\omega}(S)$.
Clearly, $q_S^*$ is an isometric embedding which maps the closed subspace of $\bigl(C_b^{k,\omega}(S)\bigr)^*$ generated by $\delta$-functionals of points in $S$ isometrically onto $G_b^{k,\omega}(S)\subset \bigl(C_b^{k,\omega}(\RR^n)\bigr)^*$. We define
\begin{equation}\label{proj}
P:=q_S^*\circ T^*.
\end{equation}
By the definition of $T_S$, for each $\delta_x^0\in G_b^{k,\omega}(\RR^n)$, $x\in\RR^n\setminus S$, and $f\in C_b^{k,\omega}(S)$ we have, for some $y_i^x\in S$,
\[
(P\delta_x^0)(f)=\delta_x^0(Tq_S f)=\sum_{i=1}^d \lambda_i^x\cdot f(y_i^x)=\left(\sum_{i=1}^d \lambda_i^x\cdot\delta_{y_i^x}^0\right)(f).
\]
Hence,
\begin{equation}\label{equ3.11}
P\delta_x^0=\sum_{i=1}^d \lambda_i^x\cdot\delta_{y_i^x}^0\quad {\rm for\ all}\quad x\in\RR^n\setminus S.
\end{equation}
Since $T\in Ext(C_b^{k,\omega}(S);C_b^{k,\omega}(\RR^n))$,
\begin{equation}\label{equ3.12}
P\delta_x^0=\delta_x^0\quad {\rm for\ all}\quad x\in S.
\end{equation}
Thus $P$ maps $G_b^{k,\omega}(\RR^n)$ into $G_b^{k,\omega}(S)$ and is identity on $G_b^{k,\omega}(S)$. Hence, $P$ is a bounded linear projection of norm $\|P\|\le \|T q_S\|\le \|T\|$.

Next, under the identification $\bigl(G_b^{k,\omega}(S)\bigr)^*=C_b^{k,\omega}(S)$ for all closed $S\subset\RR^n$ (see Theorem \ref{te1.2}), for all $x\in\RR^n\setminus S$, and $f\in C_b^{k,\omega}(S)$ we have by \eqref{equ3.11}
\[
(P^*f)(\delta_x^0)=f(P\delta_x^0)=f\left(\sum_{i=1}^d \lambda_i^x\cdot\delta_{y_i^x}^0\right)=\sum_{i=1}^d \lambda_i^x\cdot f(y_i^x)=(Tf)(\delta_x^0).
\]
The same identity is valid for $x\in S$, cf. \eqref{equ3.12}.

This implies that $P^*=T$ and completes the proof of the theorem.
\end{proof}

\section{Proofs of Theorems \ref{te1.3} and \ref{teor1.10}}
Sections 4.1 and 4.2 contain auxiliary results used in the proof of Theorem \ref{te1.3}.
\subsection{Jackson Theorem}
Recall that the {\em Jackson kernel} $J_N$ is the trigonometric polynomial of degree $2\widetilde N$, where $\widetilde N:= \left\lfloor\frac N2\right\rfloor$, given by the formula
\[
J_N(t)=\gamma_N\left(\frac{\sin\frac{\widetilde N t}{2}}{\sin \frac t2}\right)^4,
\]
where $\gamma_N$ is chosen so that $\displaystyle \int_{-\pi}^\pi J_N(t)\, dt =1$.

For a $2\pi$-periodic real function $f\in C(\RR)$
we set 
\begin{equation}\label{jack1}
(L_N f)(x):=\int_{-\pi}^\pi f(x-t)J_N(t)\, dt,\quad x\in\mathbb R.
\end{equation}
Then the classical {\em Jackson theorem} asserts (see, e.g., \cite[Ch.\,V]{T}): $L_N f$ is a real trigonometric polynomial of degree at most $N$ such that 
\begin{equation}\label{jack2}
\sup_{x\in (-\pi,\pi)}|f(x)-(L_N f)(x)|\le c\,\omega\bigl(f, \mbox{$\frac 1N$}\bigr),
\end{equation}
for a numerical constant $c>0$;
here $\omega(f,\cdot)$ is the modulus of continuity of $f$.\subsection{Convergence in the Weak$^*$ Topology of ${\mathbf C_b^{k,\omega}(\RR^n)}$} In the proof of Theorem \ref{te1.3} we use the following result.  As before, we equip $C_b^{k,\omega}(\RR^n)$ with the weak$^*$ topology induced by means of functionals in $G_b^{k,\omega}(\RR^n)\subset \bigl(C_b^{k,\omega}(\RR^n)\bigr)^*$.
\begin{Proposition}\label{prop3.1}
A sequence $\{f_i\}_{i\in\N}\subset C_b^{k,\omega}(\RR^n)$ weak$\,^*$ converges to $f\in C_b^{k,\omega}(\RR^n)$ if and only if
\begin{itemize}
\item[(a)]
\[
\sup_{i\in\N}\|f_i\|_{C_b^{k,\omega}(\RR^n)}<\infty;
\]
\item[(b)]
For all $\alpha\in\mathbb Z_+^n$, $0\le |\alpha|\le k$, $x\in\RR^n$
\[
\lim_{i\rightarrow\infty}D^\alpha f_i(x)=D^\alpha f(x).
\]
\end{itemize}
\end{Proposition}
\begin{proof}
Without loss of generality we may assume that $f=0$.
If $\{f_i\}_{i\in\N}$ weak$^*$ converges to $0$, then (a) follows from the Banach-Steinhaus theorem and (b) from the fact that each $\delta_x^\alpha\in G_b^{k,\omega}(\RR^n)$.

Conversely, suppose that $\{f_i\}_{i\in\N}\subset C_b^{k,\omega}(\RR^n)$ satisfies (a) and (b) with $f=0$. Let $g\in G_b^{k,\omega}(\RR^n)$. According to Corollary \ref{cor2.3}, given $\varepsilon>0$ there exist $J\in\N$ and families $c_{j\alpha}\in \RR$, $x_{j\alpha}\in\RR^n$,  $1\le j\le J$, $\alpha\in \Z_+^n$, $0\le |\alpha|<k$, and
$d_{j\alpha}\in \RR$, $x_{j\alpha}, y_{j\alpha}\in\RR^n$,  $x_{j\alpha}\ne y_{j\alpha}$, $1\le j\le J$, $\alpha\in \Z_+^n$, $|\alpha|=k$, such that
\[
g=\sum_{j,\alpha}c_{j\alpha}\delta_{x_{j\alpha}}^\alpha+\sum_{j,\alpha}d_{j\alpha}\frac{\delta_{x_{j\alpha}}^\alpha-\delta_{y_{j\alpha}}^\alpha}{\omega(\|x_{j\alpha}-y_{j\alpha} \|)}+g',
\]
where
\[
\sum_{j,\alpha}|c_{j\alpha}|+\sum_{j,\alpha}|d_{j\alpha}|\le \|g\|_{G_b^{k,\omega}(\RR^n)}\quad {\rm and}\quad \|g'\|_{G_b^{k,\omega}(\RR^n)}<\frac{\varepsilon}{2M},\quad M:=\sup_{i\in\N}\|f_i\|_{C_b^{k,\omega}(\RR^n)}.
\]
Further, due to condition (b), there exists $I\in\N$ such that for all $i\ge I$,
\[
\left|f_i\left(\sum_{j,\alpha}c_{j\alpha}\delta_{x_{j\alpha}}^\alpha+\sum_{j,\alpha}d_{j\alpha}\frac{\delta_{x_{j\alpha}}^\alpha-\delta_{y_{j\alpha}}^\alpha}{\omega(\|x_{j\alpha}-y_{j\alpha} \|)}   \right)  \right|<\frac{\varepsilon}{2}.
\]
Also, for such $i$,
\[
|f_i(g')|\le \|f_i\|_{C_b^{k,\omega}(\RR^n)}\cdot \|g'\|_{G_b^{k,\omega}(\RR^n)}< M\cdot \frac{\varepsilon}{2M}=\frac{\varepsilon}{2}.
\]
Combining these inequalities we obtain for all such $i$:
\[
|f_i(g)|<\varepsilon.
\]
This shows that $\lim_{i\rightarrow\infty}f_i(g)=0$. Thus $\{f_i\}_{i\in\N}$ weak$^*$ converges to $0$, as required.
\end{proof}

\subsection{Proof of Theorem \ref{te1.3}\,(1)}
We set
\[
\mathbb K_N^n:=\bigl\{x=(x_1,\dots, x_n)\in\RR^n\, :\, \max_{1\le i\le n}|x_i|\le N\bigr\}.
\]
Let $\rho:\mathbb R^n\rightarrow [0,1]$ be a fixed $C^\infty$ function with support in the cube $\mathbb K_2^n$, equals one on the unit cube $\mathbb K_1^n$. For a natural number $\ell$ we set $\rho_\ell(x):=\rho(x/\ell)$, $x\in\mathbb R^n$. Then there exist constants $c_{k,n}$ (depending on $k$ and $n$) such that
\begin{equation}\label{rhol}
\sup_{x\in\mathbb R^n}|D^\alpha\rho_\ell(x)|\le \frac{c_{k,n}}{\ell^{|\alpha|}}\quad {\rm for\ all}\quad \alpha\in\mathbb Z_+^n\quad {\rm with }\quad   |\alpha|\le k+1.
\end{equation}

Let $f\in C_b^{k,\omega}(\mathbb R^n)$. We define a 
$8\ell\sqrt n\,$-periodic in each variable function $f_\ell$ on $\mathbb R^n$ by 
\begin{equation}\label{eq3.10}
f_\ell(v+x)=\rho_\ell(x)\cdot f(x),\qquad v+x\in 8\ell\sqrt n\cdot\mathbb Z^n+\mathbb K_{4\ell\sqrt n}^n.
\end{equation}
Note that $f_\ell$ coincides with $f$ on the cube $\mathbb K_\ell^n$.
\begin{Lm}\label{norm}
There exists a constant $C_\ell=C(\ell,k,n,\omega)$ (i.e., depending on $\ell, k,n$ and $\omega$) such that 
\[ 
\lim_{\ell\rightarrow\infty}C_{\ell}=1+c_{k,n}\cdot 4\sqrt n\cdot (k+1)\cdot\lim_{t\rightarrow\infty}\,\frac{1}{\omega(t)},
\] 
and
\[
\|f_\ell\|_{C_b^{k,\omega}(\mathbb R^n)}\le C_\ell\cdot\|f\|_{C_b^{k,\omega}(\mathbb R^n)}.
\]
\end{Lm}
\begin{proof}
We use the standard multi-index notation.
According to the general Leibniz rule, for $\alpha\in\mathbb Z_+^n$, $|\alpha|\le k$,
\[
D^\alpha f_\ell=\sum_{\nu\, :\, \nu\le\alpha}\binom{\alpha}{\nu}(D^\nu\rho_\ell)\cdot (D^{\alpha-\nu}f)\quad {\rm on}\quad \mathbb K_{4\ell\sqrt n}^n .
\]
From here and \eqref{rhol} we get for $\alpha\in \mathbb Z_+^n$, $|\alpha|\ge 1$,
\[
\begin{array}{l}
\displaystyle
\sup_{x\in\mathbb R^n}|D^\alpha f_\ell(x)|\le \|f\|_{C_b^k(\mathbb R^n)}\cdot\left(\sum_{\nu\, :\, 0< \nu\le\alpha}\binom{\alpha}{\nu}\cdot\frac{c_{k,n}}{\ell^{|\nu|}}+1\right)\medskip\\
\displaystyle \le \|f\|_{C_b^k(\mathbb R^n)}\cdot\left(c_{k,n}\cdot\left(\left(1+\frac 1\ell \right)^{|\alpha|}-1\right)+1\right)\le \|f\|_{C_b^k(\mathbb R^n)}\cdot\left( \left(1+\frac 1\ell\right)^{|\alpha|-1}\cdot \frac{c_{k,n}\cdot |\alpha|}{\ell}+1  \right).
\end{array}
\]
Hence,
\[
 \|f_\ell\|_{C_b^k(\mathbb R^n)}\le\left( \left(1+\frac 1\ell\right)^{\max\{k-1,0\}}\cdot \frac{c_{k,n}\cdot k}{\ell}+1  \right)\cdot \|f\|_{C_b^k(\mathbb R^n)}=:C_1(\ell,k,n)\cdot \|f\|_{C_b^k(\mathbb R^n)}.
\]

Similarly, for $\alpha\in\mathbb Z_+^n$, $|\alpha|=k$, and $x,y \in \mathbb K_{2\ell}^n$ using properties of $\omega$ we obtain\smallskip
\begin{equation}\label{eq2.9}
\begin{array}{lr}
\qquad\\
\displaystyle
 |D^\alpha f_\ell (x)-D^\alpha f_\ell (y)|\medskip\\
\displaystyle \le \sum_{\nu\, :\, \nu\le\alpha}\binom{\alpha}{\nu}\cdot \bigl(| D^\nu \rho_\ell(x)-D^\nu\rho_\ell(y)|\cdot |D^{\alpha-\nu}f(x)  |+| D^\nu\rho_\ell(y)|\cdot |D^{\alpha-\nu}f(x)-D^{\alpha-\nu}f(y)|\bigr)\medskip\\
\displaystyle \le \|f\|_{C_b^{k,\omega}(\mathbb R^n)}
\cdot\left(\sum_{\nu\, :\, 0<\nu\le\alpha}\binom{\alpha}{\nu}\cdot c_{k,n}\cdot\left(\frac{\|x-y\|}{\ell^{|\nu|+1}}+\frac{\|x-y\|}{\ell^{|\nu|}}\right)+\frac{c_{k,n}\cdot\|x-y\|}{\ell}+\omega(\|x-y\|)\right)\medskip\\
\displaystyle =  \|f\|_{C_b^{k,\omega}(\mathbb R^n)}\cdot\left(c_{k,n}\cdot\left(\left(1+\frac 1\ell\right)^{k+1}-1\right)\cdot \|x-y\|+\omega(\|x-y\|\right)\medskip\\
\displaystyle \le   \|f\|_{C_b^{k,\omega}(\mathbb R^n)}\cdot \omega(\|x-y\|)\cdot\left( c_{k,n}\cdot \left(1+\frac 1\ell\right)^{k}\cdot\frac{4\ell\sqrt n \cdot (k+1)}{\ell\cdot\omega\bigl(4\ell \sqrt n\bigr)}+1\right)\medskip\\
\displaystyle =:\|f\|_{C_b^{k,\omega}(\mathbb R^n)}\cdot \omega(\|x-y\|)\cdot C_2(\ell,k,n,\omega).
\end{array}
\end{equation}
Observe that
\begin{equation}\label{equ4.18}
\lim_{\ell\rightarrow\infty}C_{2}(\ell,k,n,\omega)=1+c_{k,n}\cdot 4\sqrt n\cdot (k+1)\cdot\lim_{t\rightarrow\infty}\,\frac{1}{\omega(t)}.
\end{equation}

Next, assume that $x,y\in {\rm supp}\, f_\ell$. Since the case $x,y\in \mathbb K_{2\ell}^n$ was considered above, without loss of generality we may assume that $x\in \mathbb K_{2\ell}^n$ and 
 $y\in v+\mathbb K_{2\ell}^n$ for some $v\in (8\ell\sqrt n\cdot\mathbb Z^n)\setminus\{0\}$. Then for $y':=y-v\in \mathbb K_{2\ell}^n$ we have  $D^\alpha f_\ell(y')=D^\alpha f_\ell(y)$.  Also, 
 \[
 \|x-y\|=\|x-y'-v\|\ge \|v\|-\|x-y'\|\ge 8\ell\sqrt n- 4\ell\sqrt n=4\ell\sqrt n.
 \]
 Therefore for $\alpha\in\mathbb Z_+^n$, $|\alpha|=k$,
 \[
 \begin{array}{l}
 \displaystyle
 |D^\alpha f_\ell (x)-D^\alpha f_\ell (y)|= |D^\alpha f_\ell (x)-D^\alpha f_\ell (y')|\le C_2(\ell,k,n,\omega)\cdot\|f\|_{C_b^{k,\omega}(\mathbb R^n)}\cdot \omega(\|x-y'\|)\medskip\\
 \displaystyle \le C_2(\ell,k,n,\omega)\cdot\|f\|_{C_b^{k,\omega}(\mathbb R^n)}\cdot \omega(4\ell\sqrt n)\le C_2(\ell,k,n,\omega)\cdot\|f\|_{C_b^{k,\omega}(\mathbb R^n)}\cdot \omega(\|x-y\|).
 \end{array}
 \]
 
Finally, in the case $x\in {\rm supp}\, f_\ell \cap \mathbb K_{2\ell}^n$, $y\not\in {\rm supp}\, f_\ell$, there exists a point $y'\in \mathbb K_{2\ell}^n$ lying on the interval joining $x$ and $y$ such that $f_\ell(y')=0$. Since $\|x-y'\|\le \|x-y\|$ and inequality \eqref{eq2.9} is valid for $x$ and $y'$, a similar inequality is valid for $x$ and $y$. Hence, combining the considered cases we conclude that (cf. \eqref{eq5})
\[
|f_\ell|_{C_b^{k,\omega}(\mathbb R^n)}\le C_2(\ell,k,n,\omega)\cdot\|f\|_{C_b^{k,\omega}(\mathbb R^n)}.
\]

Therefore the inequality of the lemma is valid with 
\[
C(\ell,k,n,\omega):=\max\{C_1(\ell,k,n), C_2(\ell,k,n,\omega)\},
\]
see \eqref{equ4.18}.
\end{proof}

We set 
\begin{equation}\label{eq2.10}
\lambda:=\frac{4\ell\sqrt n}{\pi}.
\end{equation}
For a natural number $N$ and $x=(x_1,\dots, x_n)\in\mathbb R^n$ we define
\begin{equation}\label{eq2.11}
(E_Nf_\ell)(x)=\int_{-\pi}^\pi\cdots\int_{-\pi}^\pi f_\ell(x_1-\lambda t_1,\dots, x_n-\lambda t_n)J_N(t_1)\cdots J_N(t_n)\, dt_1\cdots dt_n.
\end{equation}
\begin{Lm}\label{lem2.2}
Function $E_Nf_\ell$ satisfies the following properties: 
\begin{itemize}
\item[(a)] $(E_Nf_\ell)(\lambda x)$, $x\in\mathbb R^n$, is the trigonometric polynomial of degree at most $N$ in each coordinate; 
\item[(b)] 
\[
\|E_Nf_\ell\|_{C_b^{k,\omega}(\mathbb R^n)}\le C_\ell\cdot\|f\|_{C_b^{k,\omega}(\mathbb R^n)};
\]
\item[(c)] There is a constant $c_N$, $\lim_{N\rightarrow\infty} c_N=0$, such that
\[
\|f_\ell-E_Nf_\ell\|_{C_b^{k}(\mathbb R^n)}\le c_N\cdot\|f\|_{C_b^{k,\omega}(\mathbb R^n)}.
\]
\end{itemize}
\end{Lm}
\begin{proof}
(a) If we set $f_\ell^\lambda (x):=f_\ell(\lambda x)$, $x=(x_1,\dots, x_n)\in \mathbb R^n$, then 
\[
E_Nf_\ell(\lambda x)=(L_N^1\cdots L_N^nf_\ell^\lambda)(x),
\]
where $L_N^i$ is the Jackson operator \eqref{jack1} acting on univariate functions in variable $x_i$, $1\le i\le n$. This and the properties of $L_N^i$ give the required statement.\smallskip

(b) According to definition \eqref{eq2.11}, for each $\alpha\in\mathbb Z_+^n$, $|\alpha|\le k$,
\begin{equation}\label{eq2.12}
\sup_{x\in\mathbb R^N}|D^\alpha(E_Nf_\ell)(x)|=\sup_{x\in\mathbb R^n}|(E_N D^\alpha f_\ell)(x)|\le \|f_\ell\|_{C_b^k(\mathbb R^n)}\le C_\ell\cdot \|f\|_{C_b^k(\mathbb R^n)}.
\end{equation}

 In turn, if $|\alpha|=k$ and $x,y\in\mathbb R^n$, then
 \begin{equation}\label{eq2.13}
 \begin{array}{l}
 \displaystyle
 |D^\alpha (E_Nf_\ell)(x)-D^\alpha(E_Nf_\ell)(y)|\le 
 |(E_N D^\alpha f_\ell)(x)- (E_ND^\alpha f_\ell)(y)|\medskip\\
 \displaystyle \le |f_\ell|_{C_b^{k,\omega}(\mathbb R^n)}\cdot \omega(\|x-y\|)\le 
 C_\ell\cdot |f|_{C_b^{k,\omega}(\mathbb R^n)}\cdot \omega(\|x-y\|).
\end{array}
 \end{equation}
 Thus \eqref{eq2.12} and \eqref{eq2.13} give the required statement.\smallskip
 
 (c) For each $\alpha\in\mathbb Z_+^n$, $|\alpha|\le k-1$, using \eqref{eq2.11} and due to \eqref{jack2} one obtains\smallskip
 \begin{equation}\label{eq2.14}
 \\
 \begin{array}{l}
 \displaystyle
 \sup_{x\in\mathbb R^n}|D^\alpha\bigl( f_\ell -E_Nf_\ell\bigr)(x)|=\sup_{x\in\mathbb R^n}|D^\alpha f_\ell(x)-(E_ND^\alpha f_\ell)(x)|\medskip\\
 \displaystyle \le \sup_{x\in\mathbb R^n}\bigl\{|D^\alpha f_\ell(\lambda x)-(L_N^1D^\alpha f_\ell)(\lambda x)|+|\bigl(L_N^1\bigl(D^\alpha f_\ell-L_N^2D^\alpha f_\ell\bigr)\bigr)(\lambda x)|+\,\cdots \medskip\\
 \displaystyle  + |\bigl(L_N^1\cdots L_N^{n-1}\bigl(D^\alpha f_\ell-L_N^nD^\alpha f_\ell\bigr)\bigr)(\lambda x)| \bigr\}\le
 \frac{c\cdot n}{N}\cdot\|f_\ell\|_{C_b^k(\mathbb R^n)}\le \frac{c\cdot n\cdot C_\ell}{N}\cdot\|f\|_{C_b^k(\mathbb R^n)}.
 \end{array}
 \end{equation}
 Similarly, for $|\alpha|=k$, we have\smallskip
 \begin{equation}\label{eq2.15}
 \sup_{x\in\mathbb R^n}|D^\alpha(f_\ell -E_Nf_\ell)(x)|\le c\cdot n\cdot C_\ell\cdot\omega(\mbox{$\frac 1N$})\cdot |f|_{C_b^{k,\omega}(\mathbb R^n)}.
 \end{equation}
 Equations \eqref{eq2.14} and \eqref{eq2.15} imply that the required statement is valid with 
 \[
 c_N:=c\cdot n\cdot C_\ell\cdot \max\left\{\mbox{$\frac 1N ,\, \omega(\frac 1N)$}\right\}.
 \]
 \end{proof}
 \begin{proof}[Proof of Theorem \ref{te1.3}\,(1)] For $f\in C_b^{k,\omega}(\RR^n)$ we set
\begin{equation}\label{eq3.18}
L_{N,\ell}f:=E_Nf_\ell.
\end{equation}
 According to Lemma \ref{lem2.2}, $L_{N,\ell}: C_b^{k,\omega}(\RR^n)\rightarrow C_b^{k,\omega}(\RR^n)$ is a finite rank bounded linear operator of norm $\le C_\ell$.
  \begin{Lm}\label{lem3.3}
 Operators $L_{N,\ell}$ are weak$\,^*$ continuous. 
 \end{Lm}
 \begin{proof} Since $C_b^{k,\omega}(\RR^n)=\bigl(G_b^{k,\omega}(\RR^n)\bigr)^*$ and $G_b^{k,\omega}(\RR^n)$ is separable, $C_b^{k,\omega}(\RR^n)$ equipped with the weak$^*$ topology is a Frechet space. Then  $L_{N,\ell}$ is weak$^*$ continuous if and only if for each sequence $\{f_i\}_{i\in\N}\subset C_b^{k,\omega}(\RR^n)$ weak$^*$ converging to  $0\in C_b^{k,\omega}(\RR^n)$ the sequence $\{L_{N,\ell}f_i\}_{i\in\N}$ weak$\,^*$ converges to $0$ as well. Note that such a sequence $\{f_i\}_{i\in\N}$ is bounded in $C_b^{k,\omega}(\RR^n)$ due to the Banach-Steinhaus theorem.
 Then $\{L_{N,\ell}f_i\}_{i\in\N}$ is bounded as well and according to Proposition \ref{prop3.1} we must prove only that
 \begin{equation}\label{eq3.19}
 \lim_{i\rightarrow\infty}D^\alpha (L_{N,\ell}f_i)(x)=0\quad {\rm for\ all}\quad \alpha\in\Z_+^n,\ 0\le |\alpha|\le k,\ x\in\RR^n.
 \end{equation}
 Further, since $D^\alpha (L_{N,\ell}f_i)(x)=(E_ND^\alpha (f_i)_\ell)(x)$ for such $\alpha$ and $x$, $\{D^\alpha (f_i)_\ell\}_{i\in\N}$ is a bounded sequence of continuous functions and $E_N$ is the convolution operator with the absolutely integrable kernel, to establish \eqref{eq3.19} it suffices to prove (due to the Lebesgue dominated convergence theorem) that
 \[
 \lim_{i\rightarrow\infty}D^\alpha (f_i)_\ell(x)=0 \quad {\rm for\ all}\quad \alpha\in\Z_+^n,\ 0\le |\alpha|\le k,\ x\in\RR^n.
 \]
 The latter follows directly from the definition of $(f_i)_\ell$, see \eqref{eq3.10}, the general Leibniz rule and the fact that $D^\alpha f_i(x)\rightarrow 0$ as $i\rightarrow\infty$ for all the required $\alpha$ and $x$ (because $\{f_i\}_{i\in\N}$ weak$^*$ converges to $0$). Thus we have proved that operators $L_{N,\ell}$ are weak$^*$ continuous.
 \end{proof}
 Lemma \ref{lem3.3} implies that there exists a bounded operator of finite rank $H_{N,\ell}$ on $G_b^{k,\omega}(\RR^n)$ whose adjoint $H_{N,\ell}^*$ coincides with $L_{N,\ell}$.
 \begin{Lm}\label{lem3.5}
 The sequence of finite rank bounded operators $\{H_{N,N}\}_{N\in\N}$ converges pointwise to the identity operator on $G_b^{k,\omega}(\RR^n)$.
 \end{Lm}
 \begin{proof}
 Let $g\in G_b^{k,\omega}(\RR^n)$. Due to Corollary \ref{cor2.3}, given $\varepsilon>0$ there exist $J\in\N$ and families $c_{j\alpha}\in \RR$, $x_{j\alpha}\in\RR^n$,  $1\le j\le J$, $\alpha\in \Z_+^n$, $0\le |\alpha|<k$, and
$d_{j\alpha}\in \RR$, $x_{j\alpha}, y_{j\alpha}\in\RR^n$,  $x_{j\alpha}\ne y_{j\alpha}$, $1\le j\le J$, $\alpha\in \Z_+^n$, $|\alpha|=k$, such that
\[
g=\sum_{j,\alpha}c_{j\alpha}\delta_{x_{j\alpha}}^\alpha+\sum_{j,\alpha}d_{j\alpha}\frac{\delta_{x_{j\alpha}}^\alpha-\delta_{y_{j\alpha}}^\alpha}{\omega(\|x_{j\alpha}-y_{j\alpha} \|)}+g''=:g'+g'',
\]
where
\[
\sum_{j,\alpha}|c_{j\alpha}|+\sum_{j,\alpha}|d_{j\alpha}|\le \|g\|_{G_b^{k,\omega}(\RR^n)}\quad {\rm and}\quad \|g''\|_{G_b^{k,\omega}(\RR^n)}<\frac{\varepsilon}{2(C_N+1)},
\]

see Lemma \ref{norm} for the definition of $C_N$.

Next, for each $f\in C_b^{k,\omega}(\RR^n)$, $\|f\|_{C_b^{k,\omega}(\RR^n)}=1$, we have by means of Lemma \ref{lem2.2},\smallskip
\begin{equation}\label{eq3.20}
\begin{array}{l}
\ \  \ \vspace*{-2mm}\\
\displaystyle
\bigl|f\bigl(H_{NN}\,g-g\bigr)\bigr|=\bigl|\bigl(L_{NN}f-f\bigr)(g)\bigr|\le \bigl|\bigl(L_{NN}f-f\bigr)(g')\bigr|+ \bigl|\bigl(L_{NN}f-f\bigr)(g'')\bigr|\medskip\\
\displaystyle < \bigl|\bigl(E_{N}f_N-f\bigr)(g')\bigr|+\|L_{NN}-{\rm id}\|\cdot\frac{\varepsilon}{2(C_{N}+1)}\le\bigl|\bigl(E_{N}f_N-f\bigr)(g')\bigr|+\frac{\varepsilon}{2}.
\end{array}
\end{equation}
Let $N_0\in\N$ be so large that all points $x_{j\alpha}, y_{j\alpha}$ as above belong to $\mathbb K_{N_0}^n$. Since $f_{N_0}=f$ on $\mathbb K_{N_0}^n$, for all $N\ge N_0$,
\[
\bigl(E_{N}f_N-f\bigr)(g')=\bigl(E_{N}f_N-f_N\bigr)(g').
\]
Hence, due to Lemma \ref{lem2.2}\,(c) for $z=x_{j\alpha}$ or $y_{j\alpha}$,
\[
\bigl|\bigl(E_{N}f_N-f\bigr)(\delta_z^\alpha)\bigr|\le \|E_{N}f_N-f_N\|_{C_b^{k}(\RR^n)} \le c_N.
\]
This implies that for all $N\ge N_0$
\[
\bigl|\bigl(E_{N}f_N-f\bigr)(g')\bigr|\le c_N\cdot\left(\sum_{j,\alpha}|c_{j\alpha}| +\left(\sum_{j,\alpha}2|d_{j\alpha}|\right)\cdot\max_{j,\alpha}\left\{\frac{1}{\omega(\|x_{j\alpha}-y_{j\alpha} \|)}\right\}\right).
\]
 Choose $N_0'\ge N_0$ so large that for all $N\ge N_0'$ the right-hand side of the previous inequality is less than $\frac{\varepsilon}{2}$. Then combining this with \eqref{eq3.20} we get for all $N\ge N_0'$,
 \[
 \|H_{NN}\,g-g\|_{G_b^{k,\omega}(\RR^n)}<\varepsilon.
 \]
 This shows that for all $g\in G_b^{k,\omega}(\RR^n)$
 \[
 \lim_{N\rightarrow\infty}H_{NN}\,g=g
 \]
which completes the proof of the lemma.
\end{proof}

Let us finish the proof of the theorem for $S=\RR^n$.  We set
\begin{equation}\label{equ4.28}
T_{N}:=\left(1+c_{k,n}\cdot 4\sqrt n\cdot (k+1)\cdot\lim_{t\rightarrow\infty}\,\frac{1}{\omega(t)}\right)\cdot\frac{H_{NN}}{C_N},
\end{equation}
see Lemma \ref{norm} for the definition of $C_N$.
Since $\{C_N\}_{N\in\N}$ converges to the first factor in the definition of $T_N$, due to Lemma \ref{lem3.5}  $\{T_N\}_{N\in\N}$ is the sequence of operators of finite rank on $G_b^{k,\omega}(\RR^n)$ of norm at most $\lambda:=1+c_{k,n}\cdot 4\sqrt n\cdot (k+1)\cdot\lim_{t\rightarrow\infty}(1/\omega(t))$ converging pointwise to the identity operator. In particular, this sequence converges uniformly to the identity operator on each compact subset of $G_b^{k,\omega}(\RR^n)$. This shows that $G_b^{k,\omega}(\RR^n)$ has the  $\lambda$-approximation property with respect to the approximating sequence of operators $\{T_N\}_{N\in\N}$.

The proof of Theorem \ref{te1.3} for $S=\RR^n$ is complete.
\end{proof}
\subsection{Proof of Theorem \ref{te1.3}\,(2)} 
\begin{proof}
In the case of $G_b^{k,\omega}(S)$, the required sequence of finite rank linear operators approximating the identity map is
$\bigl\{PT_N|_{G_b^{k,\omega}(S)}\bigr\}_{N\in\N}$, where $T_N$ are linear operators defined by \eqref{equ4.28} and $P: G_b^{k,\omega}(\RR^n)\rightarrow G_b^{k,\omega}(S)$ is the projection of Theorem \ref{teo1.6}. We have
\[
\|PT_N|_{G_b^{k,\omega}(S)}\|\le \|P\|\cdot\|T_N\|=:\|P\|\cdot\lambda(k,n,\omega).
\]
Choosing here $P$ corresponding to the extension operators of papers \cite{Gl} ($k=0$), \cite{BS2} ($k=1$) and \cite{Lu} ($k\ge 2$) we obtain the required result.

The proof of Theorem \ref{te1.3} is complete.
\end{proof}
\subsection{Proof of Theorem \ref{teor1.10}}
\begin{proof}
Due to the result of Pe\l czy\'nski \cite{P} there are a separable Banach space $Y$ with a norm one monotone basis $\{b_j\}_{j\in\N}$, an isomorphic embedding $T:X\rightarrow Y$ with distortion $\|T\|\cdot\|T^{-1}\|\le 4\lambda$,  and a linear projection $P:Y\rightarrow T(X)$ with $\|P\|\le 4\lambda$. For an operator $H\in\mathcal L(G_b^{k,\omega}(\RR^n);X)$ we define
\[
\widetilde H:=T\cdot H\in \mathcal L(G_b^{k,\omega}(\RR^n);Y).
\]
Then for each $x\in\RR^n$, 
\[
\widetilde H(\delta_x^0)=\sum_{j=1}^\infty \tilde h_j(x)\cdot b_j
\]
for some $\tilde h_j(x)\in\RR$, $j\in\N$.

Further, consider the family of bounded linear functionals $\{b_j^*\}_{j\in\N}\subset Y^*$ such that  $b_j^*(b_i)=\delta_{ij}$ (- the Kronecker delta) for all $i, j\in\N$. As the basis $\{b_j\}_{j\in\N}$ is monotone, $\|b_j^*\|\le 2$ for all $j\in\N$. Since $b_j^*\circ\widetilde H\in \bigl(G_b^{k,\omega}(\RR^n)\bigr)^*=C_b^{k,\omega}(\RR^n)$, the functions $\tilde h_j$, $\tilde h_j(x):=(b_j^*\circ \widetilde H)(\delta_x^0)$, $x\in\RR^n$, belong to $C_b^{k,\omega}(\RR^n)$ and 
\begin{equation}\label{equ4.32a}
\|\tilde h_j\|_{C_b^{k,\omega}(\RR^n)}\le 2\cdot\|T\|\cdot\|H\|\quad {\rm for\ all}\quad j\in\N.
\end{equation}
In particular, $(b_j^*\circ\widetilde H)(\delta_x^\alpha)=D^\alpha (b_j^*\circ\widetilde H)(\delta_x^0)=D^\alpha \tilde h_j(x)$ for all $\alpha\in\Z_+^n$, $|\alpha|\le k$, $x\in\RR^n$, $j\in\N$. This implies that for all such $\alpha$ and $x$,
\begin{equation}\label{equ4.32}
\widetilde H(\delta_x^\alpha)=\sum_{j=1}^\infty D^\alpha \tilde h_j(x)\cdot b_j.
\end{equation}

Next, since the range of $\widetilde H$ is the subset of $T(X)$, 
\[
H=T^{-1}\cdot P\cdot \widetilde H.
\]
From here and  \eqref{equ4.32} we obtain, for all $\alpha\in\Z_+^n$, $|\alpha|\le k$, $x\in\RR^n$,
\begin{equation}\label{equ4.33}
H(\delta_x^\alpha)=\sum_{j=1}^\infty D^\alpha \tilde h_j(x)\cdot (T^{-1}\cdot P)(b_j)
\end{equation}
(convergence in $X$).

Finally, we set
\begin{equation}\label{equ4.34}
h_j:=\|(T^{-1}\cdot P)(b_j)\|\cdot\tilde h\quad {\rm and}\quad v_j:=\frac{(T^{-1}\cdot P)(b_j)}{\|(T^{-1}\cdot P)(b_j)\|},\quad j\in\N.
\end{equation}
Then all $v_j\in X$ are of norm one. In turn, all $h_j\in C_b^{k,\omega}(\RR^n)$ and due to \eqref{equ4.32a}, 
\eqref{equ4.34} and the properties of $T$ and $P$ for all $j\in\N$,
\[
\|h_j\|_{C_b^{k,\omega}(\RR^n)}\le \|T^{-1}\|\cdot\|P\|\cdot\|\tilde h_j\|_{C_b^{k,\omega}(\RR^n)}\le  2\cdot \|T\|\cdot\|T^{-1}\|\cdot\|P\|\cdot\|H\|\le 32\cdot\lambda^2\cdot\|H\|.
\]
Moreover, by \eqref{equ4.33}, for all $\alpha\in\Z_+^n$, $|\alpha|\le k$, $x\in\RR^n$,
\[
H(\delta_x^\alpha)=\sum_{j=1}^\infty D^\alpha  h_j(x)\cdot v_j,
\]
as required.

The proof of Theorem \ref{teor1.10} is complete.
\end{proof}

\section{Proofs of Theorem \ref{te1.4} and Corollary \ref{cor1.10}} 
\subsection{Proof of Theorem \ref{te1.4}\,(1)}
\begin{proof}
 Let $\Lambda_{n,k}:=\{\alpha\in \Z_+^n\, :\, |\alpha|\le k\}$. We set 
\begin{equation}\label{eq4.23}
M_{n,k}:=\bigl(\Lambda_{n,k}\times\RR^n\bigr)\sqcup \bigl(\bigl(\Lambda_{n,k}\setminus\Lambda_{n,k-1}\bigr)\times\bigl((\RR^n\times\RR^n)\setminus\Delta_n\bigr)\bigr),
\end{equation}
where $\Delta_n:=\{(x,y)\in\RR^n\times\RR^n\, :\, x=y\}$.

Space $M_{n,k}$ has the natural structure of a $C^\infty$ manifold, in particular, it is a locally compact Hausdorff space.
By $C_b(M_{n,k})$ we denote the Banach space of bounded continuous functions on $M_{n,k}$ equipped with supremum norm. Let us define a linear map $\mathcal I: C_b^{k,\omega}(\RR^n)\rightarrow C_b(M_{n,k})$ by the formula\medskip
\begin{equation}
\mathcal I(f)(m)=\left\{
\begin{array}{lll}
\displaystyle D^\alpha f(x)&{\rm if}&m=(\alpha,x)\in\Lambda_{n,k}\times \RR^n\\
\\
\displaystyle \frac{D^\alpha f(x)-D^\alpha f(y)}{\omega(\|x-y\|)}&{\rm if}&m=(\alpha, (x,y))\in \bigl(\Lambda_{n,k}\setminus\Lambda_{n,k-1}\bigr)\times\bigl(\RR^n\times\RR^n\setminus\Delta_n\bigr),
\end{array}
\right.
\end{equation}
\begin{Proposition}\label{prop5.1}
$\mathcal I$ is a linear isometric embedding.
\end{Proposition}
\begin{proof}
The statement follows straightforwardly from the definitions of the involved spaces. 
\end{proof}
Since $\mathcal I\bigl(C_b^{k,\omega}(\RR^n)\bigr)$ is a closed subspace of $C_b(M_{n,k})$, the Hahn-Banach theorem implies that the adjoint map 
\begin{equation}\label{eq4.25}
\mathcal I^*:\bigl(C_b(M_{n,k})\bigr)^*\rightarrow \bigl(C_b^{k,\omega}(\RR^n)\bigr)^*
\end{equation}
of $\mathcal I$ is surjective of norm one.
%

Similarly, $\mathcal I$ maps $C_0^{k,\omega}(\RR^n)$ isometrically into the Banach subspace $C_0(M_{n,k})\subset C_b(M_{n,k})$ of continuous functions on $M_{n,k}$ vanishing at infinity. Thus the adjoint of $\mathcal I_0:=\mathcal I |_{C_0^{k,\omega}(\RR^n)}$ is the surjective map of norm one
\begin{equation}\label{eq4.26}
\mathcal I_0^*:\bigl(C_0(M_{n,k})\bigr)^*\rightarrow \bigl(C_0^{k,\omega}(\RR^n)\bigr)^*.
\end{equation}
According to the Riesz representation theorem (see, e.g.,\cite{DS}), $\bigl(C_0(M_{n,k})\bigr)^*$ is isometrically isomorphic to the space of countably additive regular Borel measures on $M_{n,k}$ with the norm being the total variation of measure. In what follows we identify these two spaces.

In the proof we use the following result.
\begin{Proposition}\label{prop4.2}
If $\omega$ satisfies condition \eqref{omega2}, then  $C_0^{k,\omega}(\RR^n)$ is weak$^*$ dense in $C_b^{k,\omega}(\RR^n)$.
\end{Proposition}
\begin{proof}
Let $\{L_{NN}\}_{N\in\N}$ be finite rank bounded linear operators on $C_b^{k,\omega}(\RR^n)$ defined by \eqref{eq3.18}. According to Lemma \ref{lem3.5} for each $f\in C_b^{k,\omega}(\RR^n)$ the sequence $\{L_{NN}f\}_{N\in\N}$ weak$^*$ converges to $f$. Moreover, each $L_{NN}f\in C^\infty(\RR^n)$, cf. Lemma \ref{lem2.2}\,(a). We set
\begin{equation}\label{e4.27}
\hat f_N:=\rho_N\cdot L_{NN}f,
\end{equation}
see section~4.3. Then $\hat f_N$ is a $C^\infty$ function with compact support on $\RR^n$ satisfying, due to Lemmas \ref{norm} and \ref{lem2.2}\,(b), the inequality
\begin{equation}\label{e4.28}
\|\hat f_N\|_{C_b^{k,\omega}(\RR^n)}\le C_N^2\|f\|_{C_b^{k,\omega}(\RR^n)},
\end{equation}
where $\lim_{N\rightarrow\infty}C_N=1+c_{k,n}\cdot 4\sqrt n \cdot (k+1)\cdot\lim_{t\rightarrow\infty}\frac{1}{\omega(t)}$.

Clearly, sequence $\{D^\alpha \hat f_N\}_{N\in\N}$  converges pointwise to $D^\alpha f$ for all $\alpha\in \Z_+^n$, $|\alpha|\le k$. Also, due to condition \eqref{omega2}  all $\hat f_N\in C_0^{k,\omega}(\RR^n)$. Hence, according to Proposition \ref{prop3.1}, sequence $\{\hat f_N\}_{N\in\N}$ weak$^*$ converges to $f$. This shows that $C_0^{k,\omega}(\RR^n)$ is weak$^*$ dense in $C_b^{k,\omega}(\RR^n)$.
\end{proof}

Next, let $i^*: \bigl(C_b^{k,\omega}(\RR^n)\bigr)^*\rightarrow \bigl(C_0^{k,\omega}(\RR^n)\bigr)^*$ be the linear surjective map of norm one adjoint to the isometrical embedding 
$i: C_0^{k,\omega}(\RR^n)\hookrightarrow C_b^{k,\omega}(\RR^n)$.
\begin{C}\label{cor4.3}
Restriction of $i^*$ to $G_b^{k,\omega}(\RR^n)$ is injective.
\end{C}
\begin{proof}
Proposition \ref{prop4.2} implies that functions in $C_0^{k,\omega}(\RR^n)$ regarded as linear functionals on $G_b^{k,\omega}(\RR^n)$ separate the points of $G_b^{k,\omega}(\RR^n)\, (\subset  \bigl(C_b^{k,\omega}(\RR^n)\bigr)^*)$. If $i^*(v)=0$ for some $v\in G_b^{k,\omega}(\RR^n)$, then 
\[
0=(i^*(v))(f)=f(v)\quad {\rm for\ all}\quad f\in C_0^{k,\omega}(\RR^n).
\]
Hence, $v=0$.
\end{proof}
We set
\[
\tilde\delta_x^\alpha:=i^*(\delta_x^\alpha),\quad |\alpha|\le k,\ x\in\RR^n.
\]
By definition, maps $\phi_\alpha: \RR^n\rightarrow \bigl(C_b^{k,\omega}(\RR^n)\bigr)^*$, $x\mapsto\delta_x^\alpha$, $|\alpha|\le k$, are continuous and bounded and so are the maps $i^*\circ\phi_\alpha:\RR^n\rightarrow \bigl(C_0^{k,\omega}(\RR^n)\bigr)^*$. 
\begin{Proposition}\label{prop4.4}
The range of $\mathcal I_0^*$ coincides with $i^*(G_b^{k,\omega}(\RR^n))$.
\end{Proposition}
\begin{proof}
Let $\mu\in \bigl(C_0(M_{n,k})\bigr)^*$ be a countably additive regular Borel measures on $M_{n,k}$.   We set, for all admissible $\alpha$ and all Borel measurable sets $U\subset M_{n,k}$,
\[
\mu_\alpha^1(U)=\mu\bigl(U\cap (\{\alpha\}\times\RR^n)\bigr)\quad {\rm and}\quad
\mu_\alpha^2(U)=\mu\bigl(U\cap \bigl(\{\alpha\}\times \bigl((\RR^n\times\RR^n)\setminus\Delta_n\bigr)\bigr)\bigr).
\] 
Then $\mu=\sum_{\alpha,j}\mu_\alpha^j$.

Let us show that each $\mathcal I_0^*(\mu_\alpha^j)$ belongs to $i^*(G_b^{k,\omega}(\RR^n))$. Indeed, for
$j=1$ consider the Bochner integral
\begin{equation}\label{eq4.27}
J(\mu_\alpha^1):=\int_{x\in\RR^n}i^*(\phi_\alpha(x))\, d\mu_\alpha^1(x)=i^*\left(\int_{x\in\RR^n}\phi_\alpha(x)\, d\mu_\alpha^1(x)\right).
\end{equation}
Since $\phi_\alpha$ is continuous and bounded, the above integral is well-defined and its value is an element of $i^*(G_b^{k,\omega}(\RR^n))$. By the definition of the Bochner integral, for each $f\in C_0^{k,\omega}(\RR^n)$,
\[
(J(\mu_\alpha^1))(f)=\int_{x\in\RR^n}\bigl(i^*(\phi_\alpha(x))\bigr)(f)\,d\mu_\alpha^1(x)=\int_{x\in\RR^n} D^\alpha f(x)\,d\mu_\alpha^1(x)=:\bigl(\mathcal I_0^*(\mu_\alpha^1)\bigr)(f).
\]
Hence,
\[
\mathcal I_0^*(\mu_\alpha^1)=J(\mu_\alpha^1)\in i^*(G_b^{k,\omega}(\RR^n)).
\]
Similarly, for $\alpha\in\Lambda_{n,k}\setminus\Lambda_{n,k-1}$ we define
\begin{equation}\label{eq4.28}
J(\mu_\alpha^2):=\int_{z=(x,y)\in (\RR^n\times\RR^n)\setminus\Delta_n}\frac{i^*(\phi_\alpha(x))-i^*(\phi_\alpha(y))}{\omega(\|x-y\|)} d\mu_\alpha^2(z).
\end{equation}
Then, as before, we obtain that 
\[
\mathcal I_0^*(\mu_\alpha^2)=J(\mu_\alpha^2)\in i^*(G_b^{k,\omega}(\RR^n)).
\]

Thus we have established that the range of $\mathcal I_0^*$ is a subset of $i^*(G_b^{k,\omega}(\RR^n))$. Since the map $\mathcal I_0^*$ is surjective and its range contains all $\tilde\delta_x^\alpha$, it must contain $i^*(G_b^{k,\omega}(\RR^n))$ as well.

This completes the proof of the proposition.
\end{proof}

In particular, we obtain that $\bigl(C_0^{k,\omega}(\RR^n)\bigr)^*=i^*(G_b^{k,\omega}(\RR^n))$, i.e., by the inverse mapping theorem $i^*$ restricted to $G_b^{k,\omega}$ maps it isomorphically onto $\bigl(C_0^{k,\omega}(\RR^n)\bigr)^*$. 

Let us show that if $\lim_{t\rightarrow\infty}\omega(t)=\infty$, then $i^*$ is an isometry.
Assume, on the contrary, that for some  $v\in G_b^{k,\omega}(\RR^n)$, 
\begin{equation}\label{eq4.31}
\|i^*(v)\|_{(C_0^{k,\omega}(\RR^n))^*}<\|v\|_{G_b^{k,\omega}(\RR^n)}.
\end{equation}
Let $f\in C_0^{k,\omega}(\RR^n)$, $\|f\|_{C_b^{k,\omega}(\RR^n)}=1$, be such that
\[
v(f)=\|v\|_{G_b^{k,\omega}(\RR^n)}.
\]
Let $\{f_N\}_{N\in\N}\subset C_0^{k,\omega}(\RR^n)$, 
$\|f_N\|_{C_b^{k,\omega}(\RR^n)}\le C_N^2$, $N\in\N$, be the sequence of Proposition \ref{prop4.2} weak$^*$ converging to $f$. Observe that $\lim_{N\rightarrow\infty} C_N=1$ due to the above condition for $\omega$. Then from \eqref{eq4.31} and \eqref{e4.28} we obtain
\[
\begin{array}{l}
\displaystyle
\|v\|_{G_b^{k,\omega}(\RR^n)}=v(f)=\lim_{N\rightarrow\infty} v(f_N)=\lim_{N\rightarrow\infty} \bigl(i^*(v)\bigr)(f_N)\medskip\\
\displaystyle
\le \|i^*(v)\|_{(C_0^{k,\omega}(\RR^n))^*}\cdot\varlimsup_{N\rightarrow\infty} \|f_N\|_{C_b^{k,\omega}(\RR^n)}\le \|i^*(v)\|_{(C_0^{k,\omega}(\RR^n))^*}<\|v\|_{G_b^{k,\omega}(\RR^n)},
\end{array}
\]
a contradiction proving that $i^*$ is an isometry.

The proof of Theorem \ref{te1.4}\,(1) is complete.
\end{proof}
\subsection{Proof of Theorem \ref{te1.4}\,(2)}
\begin{proof}
By the hypotheses of the theorem there exists a weak$^*$ continuous operator $T\in Ext(C_b^{k,\omega}(S);C_b^{k,\omega}(\RR^n))$ such that $T(C_0^{k,\omega}(S))\subset C_0^{k,\omega}(\RR^n)$. This implies that there is a bounded linear projection of the geometric preduals of the corresponding spaces
$P: G_b^{k,\omega}(\RR^n)\rightarrow G_b^{k,\omega}(S)$  such that $P^*=T$.
Let $q_S: C_b^{k,\omega}(\RR^n)\rightarrow C_b^{k,\omega}(S)$ and $q_{S0}: C_0^{k,\omega}(\RR^n)\rightarrow C_0^{k,\omega}(S)$ denote the quotient maps induced by restrictions of functions on $\RR^n$ to $S$.  Finally, let $i: C_0^{k,\omega}(\RR^n)\rightarrow C_b^{k,\omega}(\RR^n)$ and
$i_S: C_0^{k,\omega}(S)\rightarrow C_b^{k,\omega}(S)$ be the bounded linear maps corresponding to inclusions of the spaces. Note that $i$ is an isometric embedding and $i_S$ is injective of norm $\le 1$.
\begin{Lm}\label{lem5.5}
$T_0:=T|_{C_0^{k,\omega}(S)}:C_0^{k,\omega}(S)\rightarrow C_0^{k,\omega}(\RR^n)$ is a bounded linear map between Banach spaces. 
\end{Lm}
\begin{proof}
For $f\in C_0^{k,\omega}(S)$ we have
\[
\|T_0f\|_{C_0^{k,\omega}(\RR^n)}=\|(T\circ i_S)(f)\|_{C_b^{k,\omega}(\RR^n)}\le\|T\|\cdot\|i_S(f)\|_{C_b^{k,\omega}(S)}\le \|T\|\cdot\|f\|_{C_0^{k,\omega}(S)},
\]
as required.
\end{proof}

Now, we have the following two commutative diagrams of adjoints of the above bounded linear maps (one corresponding to upward arrows and another one to downward arrows):
\begin{equation}\label{equ5.38}
\begin{array}{cccc}
 \bigl(C_b^{k,\omega}(\RR^n)\bigr)^*&\stackrel{i^*}{\longrightarrow}&\bigl(C_0^{k,\omega}(\RR^n)\bigr)^*\smallskip\\
 _{\mbox{{\tiny $q_S^*$}}}\!\uparrow\ \ \ \downarrow\mbox{{\tiny $T^*$}}&&_{\mbox{{\tiny $q_{S0}^*$}}}\!\uparrow\ \ \ \downarrow\mbox{{\tiny $T_0^*$}}
 \\
\bigl(C_b^{k,\omega}(S)\bigr)^*&\stackrel{i_S^*}{\longrightarrow}&\bigl(C_0^{k,\omega}(S)\bigr)^*.
\end{array}
\end{equation}
Here $T^*\circ q_S^*=(q_S\circ T)^*={\rm id}$ and $T_0^*\circ q_{S0}^*=(q_{S0}\circ T_0)^*={\rm id}$, maps $q_S^*$ and $q_{S0}^*$ are isometric embeddings and
map $i^*$ is surjective. 

Note that $i^*|_{G_b^{k,\omega}(\RR^n)}:G_b^{k,\omega}(\RR^n)\rightarrow \bigl(C_0^{k,\omega}(\RR^n)\bigr)^*$ is an isomorphism by the first part of the theorem. Also, by the definition of $P$ (see \eqref{proj} in section~3.3 above),
\begin{equation}\label{equ5.39}
q_S^*\circ (T^*|_{G_b^{k,\omega}(\RR^n)})=P.
\end{equation}
Let us show that the map
\[
I:=i_S^*\circ(T^*|_{G_b^{k,\omega}(S)}):G_b^{k,\omega}(S)\rightarrow \bigl(C_0^{k,\omega}(S)\bigr)^*
\]
is an isomorphism.\smallskip

(a) Injectivity of $I$: If $I(v)=0$ for some $v\in G_b^{k,\omega}(S)$, then by the commutativity of \eqref{equ5.38} and by \eqref{equ5.39},
\[
0=(q_{S0}^*\circ i_S^*)(T^*v)=(i^*\circ q_S^*)(T^*v)=i^*(Pv)=i^*(v).
\]
Since $i^*$ is injective, the latter implies that $v=0$, i.e., $I$ is an injection.\medskip

(b) Surjectivity of $I$: Let $v\in \bigl(C_0^{k,\omega}(S)\bigr)^*$. Since $T_0^*$ is surjective, there exists $v_1\in \bigl(C_0^{k,\omega}(\RR^n)\bigr)^*$ such that $T_0^*(v_1)=v$. Further, since
$i^*|_{G_b^{k,\omega}(\RR^n)}:G_b^{k,\omega}(\RR^n)\rightarrow \bigl(C_0^{k,\omega}(\RR^n)\bigr)^*$ is an isomorphism, there exists $v_2\in G_b^{k,\omega}(\RR^n)$ such that $i^*(v_2)=v_1$.
Now, by the commutativity of \eqref{equ5.38},
\[
v=(T_0^*\circ i^*)(v_2)=(i_S^*\circ T^*)(v_2)=(i_S^*\circ (T^*\circ q_S^*)\circ T^*)(v_2)=(i_S^*\circ T^*)(Pv_2)=I(Pv_2),
\]
i.e., $I$ is a surjection.

So $I$ is a bijection and therefore by the inverse mapping theorem it is an isomorphism.

This completes the proof of the second part of Theorem \ref{te1.4}.
\end{proof}
\subsection{Proof of Corollary \ref{cor1.10}}
\begin{proof}
Let $X\subset C_0^{k,\omega}(\RR^n)$ be the closure of the space of $C^\infty$ functions with compact supports on $\RR^n$. Assume, on the contrary, that there exists $f\in C_0^{k,\omega}(\RR^n)\setminus X$. Then there exists a functional $\lambda\in \bigl(C_0^{k,\omega}(\RR^n)\bigr)^*$ such that
$\lambda|_X=0$ and $\lambda(f)=1$. 

Let $i^*: \bigl(C_b^{k,\omega}(\RR^n)\bigr)^*\rightarrow \bigl(C_0^{k,\omega}(\RR^n)\bigr)^*$ be the adjoint of the isometrical embedding $i:C_0^{k,\omega}(\RR^n)\hookrightarrow C_b^{k,\omega}(\RR^n)$. According to the arguments of the proof of Theorem \ref{te1.4}, 
$i^*|_{G_b^{k,\omega}(\RR^n)}:G_b^{k,\omega}(\RR^n)\rightarrow \bigl(C_0^{k,\omega}(\RR^n)\bigr)^*$ is an isomorphism. Hence, for $\tilde\lambda:=(i^*|_{G_b^{k,\omega}(\RR^n)})^{-1}(\lambda)$ we have $g(\tilde\lambda)=0$ for all $g\in X$ and $f(\tilde\lambda)=1$. Observe that $X$ is weak$^*$ dense in $C_0^{k,\omega}(\RR^n)$ (see the proof of Proposition \ref{prop4.2}). Thus $X$ separates the points of $G_b^{k,\omega}(\RR^n)$.
Since $g(\tilde\lambda)=0$ for all $g\in X$, the latter implies that $\tilde\lambda=0$, a contradiction with $f(\tilde\lambda)=1$. 

This shows that $X=C_0^{k,\omega}(\RR^n)$. 

Clearly, $X$ is separable (it contains, e.g., the dense countable set of functions of the form $\rho_N\cdot p$, $N\in\N$, where $p$ are polynomials with rational coefficients and $\{\rho_N\}_{N\in\N}$ is a fixed sequence of $C^\infty$ cut-off functions weak$^*$ converging in $C_b^{k,\omega}(\RR^n)$ to the constant function $f= 1$).

This  completes the proof of the corollary.
\end{proof}
\section{Proof of Theorem \ref{te1.11}}
\subsection{Proof of Theorem \ref{te1.11} for Weak $k$-Markov Sets}
First, we recall some results proved in \cite{BB1,BB2, B}.

(1) If $S\in {\rm Mar}_k^*(\RR^n)$, then a function $f\in C_b^{k,\omega}(S)$ has derivatives of order $\le k$ at each weak $k$-Markov point $x\in S$, i.e., there exists a (unique) polynomial $T_x^k(f)\in\cP_{k,n}$ such that
\[
\lim_{y\to x}\frac{|f(y)-T_x^k(f)(y)|}{\|y-x\|^k}=0.
\]

If $T_x^k(f)(z):=\sum_{|\alpha|\le k} \frac{c_\alpha}{\alpha !} (z-x)^\alpha$, $\alpha\in\Z_+^n$, then $c_\alpha$ is called the partial derivative of order $|\alpha|$ at $x$ and is denoted as $D_S^\alpha f (x)$.\medskip

(2)  If $\tilde f\in C_b^{k,\omega}(\RR^n)$ is such that $\tilde f|_S=f$, then the Taylor polynomial $T_x^k(\tilde f)$ of order $k$ of $\tilde f$ at $x$ coincides with $T_x^k(f)$.\medskip

(3) The analog of the the classical Whitney-Glaeser theorem holds:

A function $f\in C(S)$ belongs to $C_b^{k,\omega}(S)$  if and only if it has derivatives of order $\le k$ at each weak $k$-Markov point $x\in S$ and  there exists a constant $\lambda>0$ such that for all weak $k$-Markov points $x,y\in S$, $z\in\{x,y\}$ 
\begin{equation}\label{equ6.40}
\begin{array}{l}
\displaystyle
\max_{|\alpha|\le k}|D^\alpha_S f(x)|\le \lambda\quad\text{and}\\
\\ \displaystyle \max_{|\alpha|\le k}\frac{|D_S^\alpha \bigl(T_x^k(f)-T_y^k(f)\bigr)(z)|}{\|x-y\|^{k-|\alpha|}}\le\lambda\cdot\omega(\|x-y\|).
\end{array}
\end{equation} 
Moreover, 
\[
\|f\|_{C_b^{k,\omega}(S)}\approx\inf\lambda
\]
with constants of equivalence depending only on $k$ and $n$.\medskip

(4)
There exists a bounded linear extension operator $T: C_b^{k,\omega}(S)\to C_b^{k,\omega}(\RR^n)$ of finite depths 
\[
(Tf)(x):=
\left\{
\begin{array}{ccc}
\displaystyle \sum_{i=1}^\infty \lambda_i(x)f(x_i)&{\rm if}&x\in\RR^n\setminus S
\\
f(x)&{\rm if}& x\in S,
\end{array}
\right.
\]
where all $\lambda_i\in C^\infty(\RR^n)$ and have compact supports in $\RR^n\setminus S$, all $x_i\in S$ and for each $x\in\RR^n$ the number of nonzero terms in the above sum is at most ${n+k\choose n}\cdot w$, where $w$ is the order of the Whitney cover of $\RR^n\setminus S$.

The construction of $T$ repeats that of the Whitney-Glaeser extension operator \cite{Gl}, where instead of jets $T_x^k(f)$ of $f\in C_b^{k,\omega}(S)$ at  weak $k$-Markov points $x\in S$ (forming a dense subset of $S$) one uses polynomials of degree $k$ interpolating $f$ on certain subsets of cardinality ${n+k\choose n}$ close to $x$. In particular, as in the case of the Whitney-Glaeser extension operator, we obtain that $T\in Ext(C_b^{k,\omega}(S);C_b^{k,\omega}(\RR^n))$ for all moduli of continuity $\omega$. Also, by the construction, if  $f\in C_b^{k,\omega}(S)$ is the restriction to $S$ of a $C^\infty$ function with compact support on $\RR^n$, then $Tf\in C_b^{k,\omega}(\RR^n)$ has compact supports in all closed $\delta$-neihbourhoods of $S$ (i.e., sets $[S]_\delta:=\{x\in\RR^n\, :\, \inf_{y\in S}\|x-y\|\le\delta\}$, $\delta>0$).
\begin{proof}[Proof of Theorem \ref{te1.11} for $S\in {\rm Mar}_k^*(\RR^n)$]
Let $\rho\in C^\infty(\RR^n)$, $0\le \rho\le 1$, be such that $\rho|_{[S]_{1}}=1$, $\rho|_{\RR^n\setminus [S]_{3}}=0$
and for some $C_{k,n}\in\RR_+$ (depending on $k$ and $n$ only)
\begin{equation}\label{rho1}
\sup_{x\in\mathbb R^n}|D^\alpha\rho(x)|\le C_{k,n}\quad {\rm for\ all}\quad \alpha\in\mathbb Z_+^n.
\end{equation}
(E.g., such $\rho$ can be obtained by the convolution of the indicator function of $[S]_2$ with a fixed radial $C^\infty$ function with support in the unit Euclidean ball of $\RR^n$ and with $L_1(\RR^n)$ norm one.) We define a new extension operator by the formula
\begin{equation}\label{eq6.42}
\widetilde Tf=\rho\cdot Tf,\qquad f\in C_b^{k,\omega}(S).
\end{equation}
\begin{Lm}\label{lem6.1}
Operator $\widetilde T\in Ext(C_b^{k,\omega}(S);C_b^{k,\omega}(\RR^n))$ for all moduli of continuity $\omega$.
\end{Lm}
\begin{proof}
We equip $C_b^{k,\omega}(\RR^n)$ with equivalent norm \[
\|f\|_{C_b^{k,\omega}(\RR^n)}':=\max\left\{\|f\|_{C_b^{k,\omega}(\RR^n)},|f|_{C_b^{k,\omega}(\RR^n)}'\right\},
\]
where $|f|_{C_b^{k,\omega}(\RR^n)}'$ is defined similarly to $|f|_{C_b^{k,\omega}(\RR^n)}$ but with the supremum taken over all $x\ne y$ such that $\|x-y\|\le 1$, see \eqref{eq3}--\eqref{eq5}. (Note that the constants of equivalence between these two norms depend on $\omega$.)
Now, using word-by-word the arguments of Lemma \ref{norm} with $\rho_\ell$ replaced by $\rho$, $\ell$  replaced by $1$, and $c_{k,n}$ replaced by $C_{k,n}$ we obtain for some constant $C=C(k,n,\omega)$ and all $h\in C_b^{k,\omega}(\RR^n)$,
\begin{equation}\label{eq6.43}
\|\rho\cdot h\|_{C_b^{k,\omega}(\RR^n)}'\le C\cdot\|h\|_{C_b^{k,\omega}(\RR^n)}'.
\end{equation}
Since $T\in Ext(C_b^{k,\omega}(S);C_b^{k,\omega}(\RR^n))$ for all moduli of continuity $\omega$, inequality \eqref{eq6.43} implies the required statement.
\end{proof}

Clearly, $\widetilde T$ is of finite depth. Moreover, if   $f\in C_b^{k,\omega}(S)$ is the restriction to $S$ of a $C^\infty$ function with compact support on $\RR^n$, then $\widetilde Tf\in C_b^{k,\omega}(\RR^n)$ and has compact support on $\RR^n$ due to the properties of operator $T$, (see part (4) above). Finally, since the set of all $C_b^{k,1}(S)$ functions (i.e., for this space $\omega(t):=t$, $t\in\RR_+$) with compact supports on $S$ is dense in $C_0^{k,\omega}(S)$ (because $\omega$ satisfies condition \eqref{omega2}, see Corollary \ref{cor1.10}), the preceding property of $\widetilde T$ and Lemma \ref{lem6.1} imply that $\widetilde T(C_0^{k,\omega}(S))\subset C_0^{k,\omega}(\RR^n)$. Therefore $\widetilde T$ satisfies the hypotheses of Theorem \ref{te1.4}\,(2) (weak$^*$ continuity of $\widetilde T$ follows from Theorem \ref{teo1.6}). This implies the required statement:
$\bigl(C^{k,\omega}_0(S)\bigr)^*$ is isomorphic to $G_b^{k,\omega}(S)$ for all $\omega$ satisfying \eqref{omega2} and all weak $k$-Markov sets $S$.

Now, $G_b^{k,\omega}(S)$ has the metric approximation property due to the Grothendieck result \cite[Ch.\,I]{G} (formulated before Remark
\ref{k} of section~1.4 above) because this space has the approximation property by Theorem \ref{te1.3}. Also, $C_0^{k,\omega}(S)$ has the metric approximation property because its dual has it, see, e.g.,  \cite[Th.\,3.10]{C}.

The proof of the theorem for $S\in {\rm Mar}_k^*(\RR^n)$ is complete.
\end{proof}
\subsection{Proof of Theorem \ref{te1.11} in the General Case}
\begin{proof}
We require some auxiliary results.

Let $\widetilde\omega$ be the modulus of continuity satisfying
\begin{equation}\label{eq6.46}
\varlimsup_{t\rightarrow 0^+}\frac{\omega_o(t)}{\widetilde\omega(t)}<\infty.
\end{equation}
\begin{Lm}\label{lem6.2}
The restriction of the pullback map $H^*:C_b^k(\RR^n)\rightarrow C_b^k(\RR^n)$, $H^*f:=f\circ H$,  to $C_b^{k,\tilde\omega}(\RR^n)$ belongs to $\mathcal L(C_b^{k,\tilde\omega}(\RR^n);C_b^{k,\tilde\omega}(\RR^n))$.
\end{Lm}
\begin{proof}
We set $H=(h_1,\dots, h_n)$. Then by the hypothesis (a) of the theorem all $D^i h_j\in C_b^{k-1,\omega_o}(\RR^n)$, where $\omega_o$ satisfies \eqref{equ1.8}.

Let $f\in C_b^{k,\tilde\omega}(\RR^n)$. Then for each $\alpha\in\Z_+^n$, $|\alpha|\le k$, by the Fa\`{a} di Bruno formula, see, e.g., \cite{CS}, we obtain
\begin{equation}\label{eq6.47}
(D^\alpha (f\circ H))(x)=\sum_{0<|\lambda|\le |\alpha|} D^\lambda f(H(x))\cdot P_\lambda\left(\bigl[D^\beta H(x)\bigr]_{0<|\beta|\le |\alpha|}\right);
\end{equation}
here $P_\lambda\left(\bigl[D^\beta H(x)\bigr]_{0<|\beta|\le |\alpha|}\right)$ are polynomials of degrees $\le |\alpha|$ without constant terms with  coefficients in $\mathbb Z_+$ bounded by a constant depending on $k$ and $n$ only in variables $D^\beta h_j$, $0<|\beta|\le |\alpha|$, $1\le j\le n$. Since clearly $C_b^{0,\tilde\omega}(\RR^n)$ is a Banach algebra with respect to the pointwise multiplication of functions, to prove the lemma it suffices to check that all $D^\lambda f(H(\cdot))$ and $D^\beta h_j$ belong to $C_b^{0,\tilde\omega}(\RR^n)$.  For $|\beta|=k$ this is true because $D^\beta h_j\in C_b^{0,\omega_o}(\RR^n)\subset C_b^{0,\widetilde\omega}(\RR^n)$ by the definition of $H$ and by condition \eqref{eq6.46}, while for $1\le |\beta|\le k-1$ because $D^\beta h_j\in C_b^{0,1}(\RR^n)$ which is continuously embedded into $C_b^{0,\widetilde\omega}(\RR^n)$. Similarly, for $D^\lambda f(H(\cdot))$ with $1\le |\alpha|\le k-1$  this is true because of the continuous embedding $C_b^{0,1}(\RR^n)\hookrightarrow C_b^{0,\widetilde\omega}(\RR^n)$ and because $H$ is Lipschitz, while  for $|\lambda|=k$ by the definition of $f$ and the fact that $H$ is Lipschitz.
\end{proof}
Using this lemma we prove the following result.
\begin{Lm}\label{lem6.3}
The operator $(H|_{S'})^*: C_b^{k,\widetilde\omega}(S)\rightarrow C_b^{k,\widetilde\omega}(S')$, $(H|_{S'})^*f:=f\circ H|_{S'}$, is well-defined and belongs to $\mathcal L(C_b^{k,\widetilde\omega}(S); C_b^{k,\widetilde\omega}(S'))$. Moreover, it is weak$^*$ continuous.\footnote{Here the weak$^*$ topologies are defined by means of functionals in $G_b^{k,\widetilde\omega}(\tilde S)$, where $\tilde S$ stands for $S'$ or $S$.}
\end{Lm}
\begin{proof}
Let $\tilde f\in C_b^{k,\widetilde\omega}(\RR^n)$ be such that $\tilde f|_{S}=f$ and $\|\tilde f\|_{C_b^{k,\widetilde\omega}(\RR^n)}=\|f\|_{C_b^{k,\widetilde\omega}(S)}$. Then by Lemma \ref{lem6.2} we have
\[
\begin{array}{l}
\displaystyle
(H|_{S'})^*f=f\circ H_{S'}=(\tilde f\circ H)|_{S'}=(H^*\tilde f)|_{S'}\in C_b^{k,\widetilde\omega}(S')\quad {\rm and}\\
\\
\displaystyle \|(H|_{S'})^*f\|_{C_b^{k,\widetilde\omega}(S')}\le \|H^*\|\cdot\|\tilde f\|_{C_b^{k,\widetilde\omega}(\RR^n)}=\|H^*\|\cdot\|f\|_{C_b^{k,\widetilde\omega}(S)}.
\end{array}
\]

This shows that the operator $(H|_{S'})^*: C_b^{k,\widetilde\omega}(S)\rightarrow C_b^{k,\widetilde\omega}(S')$ is well-defined and belongs to $\mathcal L(C_b^{k,\widetilde\omega}(S); C_b^{k,\widetilde\omega}(S'))$.

Further, the fact that the operator $H^*:C_b^{k,\widetilde\omega}(\RR^n)\rightarrow C_b^{k,\widetilde\omega}(\RR^n)$ is weak$^*$ continuous follows straightforwardly from Proposition \ref{prop3.1}, Lemma \ref{lem6.2} and the the Fa\`{a} di Bruno formula \eqref{eq6.47}.
Let $T\in Ext(C_b^{k,\widetilde\omega}(S);C_b^{k,\widetilde\omega}(\RR^n))$ be the extension operator of finite depth (see section~1.3) and $q_{S'}: C_b^{k,\widetilde\omega}(\RR^n)\rightarrow C_b^{k,\widetilde\omega}(S')$ be the quotient map induced by restrictions of functions on $\RR^n$ to $S'$. Then clearly, for all $f\in C_b^{k,\widetilde\omega}(S)$,
\[
(H|_{S'})^*f=f\circ H|_{S'}=((Tf)\circ H)|_{S'}=(q_{S'}\circ H^*\circ T)f.
\]
Therefore $(H|_{S'})^*=q_{S'}\circ H^*\circ T$. Here the operator $T$ is weak$^*$ continuous by Theorem \ref{teo1.6} and the operator $q_{S'}$ is  weak$^*$ continuous because it is adjoint of the isometric embedding $G_b^{k,\widetilde\omega}(S')\hookrightarrow G_b^{k,\widetilde\omega}(\RR^n)$. This implies that the operator $(H|_{S'})^*$ is weak$^*$ continuous as well.
\end{proof}

We are ready to prove Theorem \ref{te1.11}. 

Let  $\widetilde T\in Ext(C_b^{k,\omega}(S');C_b^{k,\omega}(\RR^n))$ be the extension operator of the first part of Theorem \ref{te1.11}, see \eqref{eq6.42}. We set (for $\widetilde \omega:=\omega$)
\begin{equation}\label{eq6.48}
E:=\widetilde T\circ (H|_{S'})^*.
\end{equation}
\begin{Lm}\label{lem6.4}
Operator $E\in Ext(C_b^{k,\omega}(S); C_b^{k,\omega}(\RR^n))$, is weak$^*$ continuous and maps $C_0^{k,\omega}(S)$ in $C_0^{k,\omega}(S')$.
\end{Lm}
\begin{proof}
The first two statements follow from the hypotheses of the theorem, Lemma \ref{lem6.3} and the fact that $\widetilde T$ is weak$^*$ continuous. So let us check the last statement.

Let $f\in C_0^{k,\omega}(S)$ be the restriction of a $C^\infty$ function with compact support on $\RR^n$. Since $H|_{S'}:S'\rightarrow S$ is a proper map (by hypothesis (b) of the theorem),  $(H|_{S'})^*f\in C_b^{k,\omega}(S')$ has compact support.  Moreover, since $f\in C_b^{k,\omega_o}(S)$,  Lemma \ref{lem6.3} applied to $\widetilde\omega=\omega_o$ implies that $(H|_{S'})^*f\in C_b^{k,\omega_o}(S')$. Finally, since $\widetilde T\in Ext(C_b^{k,\omega_o}(S');C_b^{k,\omega_o}(\RR^n))$ as well, $Ef\in C_b^{k,\omega_o}(\RR^n)$ and has compact support (because $(H|_{S'})^*f$ has it). Due to condition \eqref{equ1.8} for $\omega_o$ we obtain from here that $Ef\in C_0^{k,\omega}(\RR^n)$. Since the set of such functions $f$ is dense in $C_0^{k,\omega}(S)$ (see Corollary \ref{cor1.10}), $E$ maps $C_0^{k,\omega}(S)$ in $C_0^{k,\omega}(\RR^n)$, as required.
\end{proof}

Now the result of the theorem follows from Lemma \ref{lem6.4} and Theorem \ref{te1.4}\,(2); that is,
$G_b^{k,\omega}(S)$ is isomorphic to $\bigl(C_0^{k,\omega}(S)\bigr)^*$ and so $G_b^{k,\omega}(S)$ and $C_0^{k,\omega}(S)$
have the metric approximation property (see the argument at the end of section~6.1 above).

The proof of the theorem is complete. 
\end{proof}

\end{document}